\documentclass[10pt, xcolor = {usenames,dvipsnames}]{article}

\usepackage[english]{babel}
\usepackage[utf8]{inputenc} 
\usepackage{amsfonts}
\usepackage[square,sort,comma,numbers]{natbib}
\bibpunct{[}{]}{;}{a}{,}{,}
\usepackage[hidelinks]{hyperref}
\usepackage{xfrac}
\usepackage{tikz}
\usetikzlibrary{calc}

\usepackage{mathtools}
\usepackage{amsthm}
\usepackage{amssymb}
\usepackage{amsmath}
\usepackage{needspace}
\usepackage{etoolbox}
\usepackage{lipsum}
\usepackage{authblk}
\usepackage{mathrsfs}
\usepackage{bbold}
\usepackage{graphicx}
\usepackage{subcaption}
\usepackage{color,soul}

\newcounter{theo}
\newtheorem{lmm}[theo]{Lemma}
\newtheorem{thr}{Theorem}
\newtheorem{stm}[theo]{Statement}
\newtheorem{crl}[theo]{Corollary}
\newtheorem{prp}[theo]{Proposition}

\newtheorem{defn}[theo]{Definition}
\newtheorem{rmrk}[theo]{Remark}

\newtheorem{innercustomthm}{Theorem}
\newenvironment{customthm}[1]
  {\renewcommand\theinnercustomthm{#1}\innercustomthm}
  {\endinnercustomthm}

\newtheorem{innercustomprp}{Proposition}
\newenvironment{customprp}[1]
  {\renewcommand\theinnercustomprp{#1}\innercustomprp}
  {\endinnercustomprp}

\DeclareMathOperator{\arccot}{arccot}
\DeclareMathOperator{\arccosh}{arccosh}

\DeclareMathOperator*{\argmin}{arg\,min}
\newcommand{\R}{\mathbb{R}}
\newcommand{\Compl}{\mathbb{C}}

\newcommand{\Z}{\mathbb{Z}}
\newcommand{\T}{\mathbb{T}}
\newcommand{\D}{\mathbb{D}}

\newcommand{\eps}{\varepsilon}

\newcommand\CnjCl[1]{#1}
\newcommand{\I}{\mathcal{I}}
\newcommand{\PartF}{\mathcal{Z}}
\newcommand{\PartFunc}[1]{\PartF\left(#1\right)}
\newcommand\Dyadic[1]{D_{#1}}

\newcommand{\LogSob}{\mathsf{H}}
\newcommand{\Sob}{H^1}
\newcommand{\MyHol}[2]{\mathsf{H\ddot{o}l}^{\,#1}\!\left(#2\right)}

\newcommand{\Diff}{\mathrm{Diff}}
\newcommand{\DiffG}{\widetilde{\mathrm{Diff}}}
\newcommand{\DiffAppl}{\mathsf{A}}
\newcommand{\SL}{\mathrm{PSL}}

\newcommand{\FMeas}[1]{\mathscr{M}_{#1}}
\newcommand{\FMeasSL}[1]{\widetilde{\mathscr{M}}_{#1}}

\newcommand{\Schw}{\mathcal{S}}

\newcommand{\A}{\mathsf{P}}

\renewcommand*{\d}{\mathop{}\!\mathrm{d}}
\newcommand{\WS}[3]{\mathcal{B}_{#1}^{\, #2, #3}} %
\newcommand{\Cfree}{C_{0,\text{free}}}
\newcommand{\Dom}{\Omega}
\newcommand{\Real}{\Re}

\newcommand{\vardbtilde}[1]{\widetilde{\raisebox{0pt}[0.85\height]{$\widetilde{#1}$}}}
\newcommand{\DiffGG}{\vardbtilde{\mathrm{Diff}}}

\makeatletter
\def\obs#1#2#3{\def\temp@expr{\big(#1; #2, #3\big)}\mathcal{O}\obs@}
\def\obs@{%
	\@ifnextchar{_}{\obs@sub}{
	\@ifnextchar{^}{\obs@sup}{\temp@expr}}}
\def\obs@sub#1#2{_{#2}\obs@}
\def\obs@sup#1#2{^{\, #2}\obs@}
\makeatother

\makeatletter
\def\dist#1#2{\def\temp@expr{(#1, #2)}\mathsf{d}\dist@}
\def\dist@{%
	\@ifnextchar{_}{\dist@sub}{
	\@ifnextchar{^}{\dist@sup}{\temp@expr}}}
\def\dist@sub#1#2{_{#2}\dist@}
\def\dist@sup#1#2{^{#2}\dist@}
\makeatother

\title{Large Deviations of the Schwarzian Field Theory}
\author{Ilya Losev\footnote{Mathematical Institute, University of Oxford, Andrew Wiles Building, Radcliffe Observatory Quarter, Woodstock Road, Oxford, OX2 6GG, UK.
E-mail: \url{ilya.losev@maths.ox.ac.uk}.}}

\date{May 25, 2026}

\begin{document}

\maketitle

\abstract{
We prove a large deviations principle for the probabilistic Schwarzian Field Theory at low temperatures.
We demonstrate that the good rate function is equal to the action of the Schwarzian Field Theory, and we find its minimisers.
In addition, we define an analogue of the H\"{o}lder condition on the functional space $\Diff^1(\T)/\SL(2, \R)$ in terms of cross-ratio observables, characterise them in terms of the usual H\"{o}lder property on the space of continuous functions, and deduce the corresponding compact embedding theorem.
We also show that the Schwarzian measure concentrates on functions satisfying the defined condition.
}

\section{Introduction and main results}\label{sect_introduction_mainr_results}
\subsection{Introduction}\label{sect_intro}
The Schwarzian Field Theory is a quantum field theory, which emerged in physics in the study of low-dimensional quantum gravity models.
In the context of the AdS/CFT correspondence, it is proposed as a holographic dual of Jackiw–Teitelboim (JT) gravity in the disk \citep{SaadShenkerStanford2019, NearlyAdS, JT_Wilson_Line}.
Moreover, the Schwarzian Field Theory also appears in the low energy limit of the Sachdev–Ye–Kitaev (SYK) random matrix model (e.g. see \citep{MaldacenaStanford} and \citep{KitaevJosephine}) and is connected to 2D Liouville Field Theory, infinite dimensional symplectic geometry, representation theory of the Virasoro algebra, 2D Yang-Mills, and other topics.

In \citep*{BLW}, we have defined and proved uniqueness of a finite measure on
$\mathrm{Diff}^{1}(\mathbb{T})/\mathrm{PSL}(2, \mathbb{R})$ which corresponds
to the Schwarzian Field Theory.
We have also constructed this measure and derived its partition function (i.e. total mass), following the plan proposed in \citep{BelokurovShavgulidzeExactSolutionSchwarz, BelokurovShavgulidzeCorrelationFunctionsSchwarz}. 
In another companion paper \citep{LosevCorr} we have computed correlation functions of cross-ratio observables, for which the corresponding Wilson lines are nonintersecting. 
In addition, we have shown that the obtained correlation functions characterise the measure uniquely. 
All these results were obtained using methods of stochastic analysis. 
Our results agree with the partition function formula derived in \citep{StanfordWittenFermionicLocalization} using the formal application of the Duistermaat–Heckman theorem on the infinite dimensional symplectic space $\Diff^1(\T)/\SL(2, \R)$, and the correlation function formulae obtained in \citep{ConformalBootstrap} using the conformal bootstrap and the DOZZ formula in a degenerate limit of 2D Liouville CFT.

In this work we continue studying probabilistic and analytic properties of the Schwarzian measure constructed in \citep*{BLW}.
We prove a large deviations principle, show that the corresponding good rate function is equal to the action of the Schwarzian Field Theory, and find its minimisers.
This gives a direct way to relate the Schwarzian measure to the action of the theory, which does not appear explicitly in the construction carried out in \citep*{BLW}.
Note that, alternatively we can see traces of the action in the corresponding version of Girsanov's Theorem proved in \citep*{BLW}.
However, the action does not appear there explicitly.

One of the tools developed here, which we think is of its own interest, is an analogue of the H\"{o}lder condition on the functional space $\Diff^1(\T)/\SL(2, \R)$.
The difficulty here is that elements of $\Diff^1(\T)/\SL(2, \R)$ are not functions (but rather just conjugacy classes) for which the notion of a value at a given point is not well defined. 
Thus, the usual H\"{o}lder condition for continuous functions does not make sense in this case.
We formulate a similar property in terms of cross-ratio observables, which are well defined on the quotient space, and prove that this property coincides with the usual H\"{o}lder condition if we fix the gauge (i.e., explicitly choose a representative in each conjugacy class).
This, in particular, allows us to prove a compact embedding theorem.
As an application of this tool, we show that the Schwarzian measure concentrates on functions satisfying the defined H\"{o}lder condition.

\medskip

Formally, the measure corresponding to the Schwarzian Field Theory is supported on the topological space $\Diff^1(\T)/\SL(2, \R)$ and is given by (see \citep[(1.1)]{StanfordWittenFermionicLocalization})
\begin{equation}\label{eq:1}
\d\FMeas{\sigma^2}\big(\phi\big) = 
\exp\left\{+\frac{1}{\sigma^2}\int_{0}^{1} 
\left[\Schw_\phi(\tau)+2\pi^2\phi'^{\, 2}(\tau)\right] \d\tau \right\}
\frac{\prod_{\tau \in [0,1)}\frac{\d\phi(\tau)}{\phi'(\tau)}}{\SL(2, \R)} ,
\end{equation} 
where $\Schw_{\phi}(\tau)$ is the Schwarzian derivative of $\phi$ given by
\begin{equation}
\Schw_{\phi}(\tau) = \Schw(\phi, \tau) = \left(\frac{\phi''(\tau)}{\phi'(\tau)}\right)'- \frac{1}{2} \left(\frac{\phi''(\tau)}{\phi'(\tau)}\right)^2.
\end{equation}
Here, $\T = [0, 1]/\{0\sim 1\}$ is the unit circle, $\Diff^1(\T)$ is the space of $C^1$ orientation preserving diffeomorphisms of $\T$, and $\SL(2, \R)$ is the group of M\"{o}bius transformations of the unit disk (i.e. conformal isomorphisms of the unit disk) restricted to the boundary which is identified with $\T$.
The group $\SL(2, \R)$ acts on $\Diff^1(\T)$ by post-compositions. 
Following \citep{StanfordWittenFermionicLocalization} we call it a right action, since in \citep{StanfordWittenFermionicLocalization} it is interpreted as an action on the inverse elements.
We denote the quotient of $\Diff^1(\T)$ by this action of $\SL(2, \R)$ by $\Diff^1(\T)/\SL(2, \R)$.
Heuristically, the formal density \eqref{eq:1} only depends on the orbit of this action and the quotient by $\SL(2,\R)$, therefore, makes sense.
Note, however, that even though the action and the theory are formally invariant under $\SL(2, \R)$ post-compositions, they are \textit{not} invariant under pre-compositions.
We recall the rigorous construction of the measure corresponding to the Schwarzian Field Theory carried out in \citep*{BLW} in Section~\ref{sect_meas_construct}.

Using terminology from physics, the expression in the exponential of \eqref{eq:1} is equal to minus $1/\sigma^2$ multiplied by the \emph{action}. 
In other words, the action is formally given by
\begin{equation}\label{eq_def_action_intro}
\I(\phi) 
= - \int_{0}^{1} \left[\Schw_\phi(\tau)+2\pi^2\phi'^{\, 2}(\tau)\right] \d\tau
=
\dfrac{1}{2}\displaystyle\int_{\T} 
\left[\left(\dfrac{\phi''(\tau)}{\phi'(\tau)}\right)^2-4\pi^2\phi'^{\, 2}(\tau)\right] \d\tau.
\end{equation}
Formally, it plays the defining role in the path integral formulation of the theory given by \eqref{eq:1}.
However, as it often happens with the probabilistic formulations of quantum field theories in the continuous setting, the action does not appear explicitly in the rigorous construction of the corresponding measure.
Instead, it is usually used merely as a guide for intuition.
One of the important questions in these situations is how to explicitly relate the rigorously constructed measure to its path integral formulation and, in particular, its action.
For the Schwarzian Theory it is possible to see traces of the action in the corresponding version of Girsanov's Theorem \citep*{BLW}.
In this work, we show that the action also naturally arises in the large deviations of the measure when $\sigma\to 0$, see Theorem \ref{thrLDP}. This provides an alternative way to recover the action from the constructed measures, albeit only in the asymptotic regime.
Note that large deviations principles can also be used to read off the action from probability measures in other examples, such as Gaussian path integrals (see, e.g. \citep[Theorem 3.4.12]{Deuschel_Stroock_LDP}), 2D Yang--Mills
\citep{Levy_Norris_YM_LDP}, and Schramm--Loewner Evolution
\citep{Wang_LDP, Peltola_Wang_LDP}.

We also study minimisers of the action \eqref{eq_def_action_intro} under various local constraints (Proposition \ref{prp_action_min_constrained}). 
In particular, we show that the global minimiser of \eqref{eq_def_action_intro} is the conjugacy class of the identity diffeomorphism (Theorem \ref{thrActionMin}).

One of the tools that we develop here, in order to prove the large deviations principle, is an analogue of the H\"{o}lder condition for the functional space $\Diff^1(\T)/\SL(2, \R)$ and which we apply to the Schwarzian measure $\d\FMeas{\sigma^2}$. 
From the construction of $\FMeas{\sigma^2}$, we expect that for a typical $\phi \sim \d\FMeas{\sigma^2}$, the function $\log\phi'$ should have the same regularity as Brownian motion.
In particular, it should be H\"{o}lder with parameter $\frac{1}{2}-\eps$.
However, we are working on the quotient space $\Diff^1(\T)/\SL(2, \R)$, and so $\phi$ is defined only up to composition with M\"{o}bius transformations, and thus the value of $\log\phi'$ is not well defined.
Therefore, the usual H\"{o}lder condition does not make sense in this situation.
One of the possible solutions is to fix the gauge, that is, choose a representative in every conjugacy class. 
In order to fix a representative in every $\SL(2, \R)$ orbit, it is sufficient, for example, to fix three values of $\phi$ and/or $\phi'$.
This, however, might not look like a natural definition since it is hard to express this condition directly in terms of natural observables on $\Diff^1(\T)/\SL(2, \R)$.
In this work we reformulate this property in a natural way in terms of cross-ratio observables
\begin{equation}\label{eqDefObs_intro}
\obs{\phi}{s}{t} = \frac{\pi\sqrt{\phi'(t)\phi'(s)}}{\sin\Big(\pi\big[\phi(t)-\phi(s)\big]\Big)}, \qquad s,t\in \T,
\end{equation}
which are well defined on $\Diff^1(\T)/\SL(2, \R)$ (see Section \ref{sect_meas_construct}).
More precisely, for a parameter $\alpha>0$ and a constant $K>0$, we consider $\phi \in \Diff^1(\T)/\SL(2, \R)$ for which
\begin{equation}\label{eqHolderSL_intro}
\left|\obs{\phi}{s}{t}
- \frac{\pi}{\sin \big(\pi [t-s]\big)}\right|\leq  K \, \dist{s}{t}^{\alpha-1},
\qquad \text{for all }s, t\in \T,
\end{equation}
where $\dist{\cdot}{\cdot}$ is the distance on the circle $\T$.
We demonstrate that the condition \eqref{eqHolderSL_intro} is equivalent, if we fix the gauge, to $\log \phi'$ being H\"{o}lder with parameter $\alpha$, (Theorem \ref{thrHolderEquiv}).
Finally, we show that the set of $\phi$ which do not satisfy \eqref{eqHolderSL_intro}, has exponentially small $\FMeas{\sigma^2}$ measure (Theorem \ref{thrHolderClassExpTail}).

\subsection{Main results}

\subsubsection{Large deviations principle}
The large deviations principle is usually studied in the context of probability measures.
In this work, however, we are applying it to measures $\FMeas{\sigma^2}$ which  are merely finite. 
The partition function (i.e., total mass of $\FMeas{\sigma^2}$) is equal to
\begin{equation}\label{eq_intro_part_func}
\PartF(\sigma^2) =  
    \left(\frac{2\pi}{\sigma^2}\right)^{3/2} \exp\left(\frac{2\pi^2}{\sigma^2}\right)
    =\int_0^{\infty} e^{-{\sigma^2k^2}/{2}} \sinh(2\pi k) \, 2 k \d k,
\end{equation}
as shown in \citep*{BLW}. 
We, therefore, note that, even though the large deviations principle is usually formulated only for normalised measures, it is also applicable to nonnormalised finite measures.
Therefore, throughout the paper we apply definitions and classical results from the theory to $\FMeas{\sigma^2}$.
Alternatively, one can also consider normalised measures $\PartF(\sigma^2)^{-1}\d\FMeas{\sigma^2}$ and apply the usual large deviations principle to them.
In our setting these formulations are equivalent, since by \eqref{eq_intro_part_func}
\begin{equation}
\lim_{\sigma\to 0}\sigma^2 \log \PartF(\sigma^2) = 2\pi^2,
\end{equation}
which merely causes a shift of the good rate function by $2\pi^2$.
For convenience, however, we state and prove all our results for the unnormalised measure $\FMeas{\sigma^2}$.

We will be working with the topological space $\Diff^1(\T)/\SL(2, \R)$.
Its topology is defined as the quotient topology induced by $C^1$ topology on $\Diff^1(\T)$.
Recall that in \citep*{BLW} it was showed that in this setting $\FMeas{\sigma^2}$ is a Borel measure.
\medskip

First, we define a space $\LogSob(\T)\subset\Diff^1(\T)$, on which the good rate function (or rather its lift along the quotient map $\Diff^1(\T) \twoheadrightarrow \Diff^1(\T)/\SL(2,\R)$) is finite.
\begin{defn}
Denote $\LogSob(\T) = \left\{\phi \in \Diff^1(\T) \, \big| \, \log \phi' \in \Sob(\T)\right\}$, where $\Sob(\T)$ is the Sobolev space on $\T$.
The metric on this space is defined as
\begin{equation}
\dist{\phi}{\psi}_{\LogSob} =\sqrt{\dist{\phi(0)}{\psi(0)}^2+ \left\|\log \phi' - \log \psi'\right\|_{\Sob}^2}.
\end{equation}
\end{defn}
Note that $\LogSob(\T)$ is not a linear space. 
The space $\LogSob(\T)$ is defined so that the good rate function is continuous with respect to its metric.
\begin{rmrk}
This definition is natural, since from the construction in \citep*{BLW}, we expect that if $\phi \sim \d\FMeas{\sigma^2}$, then $\log\phi'$ behaves similar to Brownian motion on small scales.
\end{rmrk}
\begin{rmrk}
As a metric space $\LogSob(\T)$ is isomorphic to $\T \times \widetilde{\Sob}(\T)$, where $\widetilde{\Sob}(\T)$ is a closed subset of $\Sob(\T)$ given by $\widetilde{\Sob}(\T) = \left\{f \in \Sob(\T) \, \big| \, \int_{\T} e^{f(t)}\d t=1\right\}$. 
This isomorphism is given by
\begin{align}
\LogSob(\T) & \to\T \times \widetilde{\Sob}(\T)\\
\phi & \mapsto \big(\phi(0),\, \log \phi'\big).
\end{align}
In particular, $\LogSob(\T)$ is a complete metric space.
It is also easy to check that $\LogSob(\T)$ is invariant under the $\SL(2, \R)$ action (although its metric is not).
\end{rmrk}

Now, we define a function, which turns out to be the lift of the good rate function along the quotient map $\Diff^1(\T) \twoheadrightarrow \Diff^1(\T)/\SL(2,\R)$.
\begin{defn}
We define the \emph{action} by 
\begin{equation}\label{eq_def_action}
\I(\phi) = 
\begin{cases} 
\dfrac{1}{2}\displaystyle\int_{\T} 
\left[\Big(\log \phi'(\tau)\Big)'^{\, 2}-4\pi^2\phi'^{\, 2}(\tau)\right] \d\tau,& \text{ for } \phi\in \LogSob(\T);\\
\infty, & \text{ otherwise}.
\end{cases}
\end{equation}
Here $\big(\log \phi'(t)\big)'$ is well defined as a weak derivative; since for $\phi\in \LogSob(\T)$, we have $\log \phi' \in \Sob(\T)$.
\end{defn}
\begin{rmrk}\label{rmrk_good_rate_for_smooth}
For any $\phi\in C^3(\T)$,
\begin{equation}
\I(\phi) = 
\dfrac{1}{2}\displaystyle\int_{\T} 
\left[\left(\dfrac{\phi''(\tau)}{\phi'(\tau)}\right)^2-4\pi^2\phi'^{\, 2}(\tau)\right] \d\tau.
\end{equation}
Moreover, since $-\Schw\big(\phi, \tau\big)$ differs from $\frac{1}{2}\left(\frac{\phi''(\tau)}{\phi'(\tau)}\right)^2$ by a total derivative $\left(\frac{\phi''(\tau)}{\phi'(\tau)}\right)'$, we have
\begin{equation}
\I(\phi) = -\displaystyle\int_{\T} 
\Big[\Schw\big(\phi, \tau\big)+2\pi^2\phi'^{\, 2}(\tau)\Big] \d\tau.
\end{equation}
\end{rmrk}
\begin{rmrk}
The action $\I(\cdot)$ is continuous on $\LogSob(\T)$ with respect to its metric $\dist{\cdot}{\cdot}_{\LogSob}$.
This follows immediately from the fact that the $\int\big(\log \phi'(\tau)\big)'^{\, 2}$ term in \eqref{eq_def_action} is precisely equal to the square of the $H^1$ norm of $\log\phi'$.
\end{rmrk}

The action $\I(\cdot)$ also defines a function on the quotient space $\Diff^1(\T)/\SL(2, \R)$ because, by the following Proposition, it is $\SL(2, \R)$ invariant.
We denote the corresponding function on the quotient space in the same way.
\begin{prp}\label{prpActionSLInvariance}
The action $\I(\cdot)$ is invariant under post-composition by elements of $\SL(2, \R)$.
In other words, for any $\psi\in \SL(2, \R)$ and any $\phi\in \LogSob(\T)$,
\begin{equation}
\I(\psi\circ \phi) = \I(\phi).
\end{equation}
\end{prp}

We now formulate the main result of this paper.
\begin{thr}\label{thrLDP}
The family of measures $\left\{\d \FMeas{\sigma^2}\right\}_{\sigma>0}$ satisfies the large deviations principle for $\sigma\to 0$ with good rate function $\I(\cdot)$.
In other words, for any $\alpha \in \R$ the level set 
\begin{equation}
\Psi(\alpha) = \left\{\phi\in \Diff(\T)/\SL(2, \R): \I(\phi)\leq \alpha \right\}
\end{equation}
is compact, and
\begin{align}
\liminf_{\sigma\to 0} \sigma^2\, \log \FMeas{\sigma^2}(U) &\geq - \inf_{\phi\in U}\I(\phi), \quad \text{if } U\subset \Diff^1(\T)/\SL(2, \R) \text{ is open in } 
C^1 \text{ topology}, \\
\limsup_{\sigma\to 0} \sigma^2\, \log \FMeas{\sigma^2}(V) &\leq - \inf_{\phi\in V}\I(\phi), \quad \text{if } V\subset \Diff^1(\T)/\SL(2, \R) \text{ is closed in }
C^1 \text{ topology}.
\end{align}
\end{thr}

Moreover, we find the minimum of the good rate function $\I(\cdot)$ and find where it is achieved.
\begin{thr}\label{thrActionMin}
For any $\phi \in \LogSob(\T)$, 
\begin{equation}\label{eqLemmaActionBoundResult}
\I(\phi)=\frac{1}{2}
\int_{0}^{1} 
\left[\left(\frac{\phi''(\tau)}{\phi'(\tau)}\right)^2-4\pi^2\phi'^{\, 2}(\tau)\right] \d\tau
\geq -2\pi^2.
\end{equation}
Moreover, the equality holds if and only if $\phi$ is a M\"{o}bius transformation of the unit circle.
\end{thr}
\begin{rmrk}
The inequality above follows from the formula of partition function and the large deviations principle. 
However, in this paper we give a more direct prove.
\end{rmrk}
Recall that the fact that $\I(\cdot)$ is a good rate function implies that it is lower semicontinuous.
Therefore, combining these two Theorems above we immediately deduce that the measure $\FMeas{\sigma^2}$ concentrates around M\"{o}bius transformations for small $\sigma$.
\begin{crl}
The probability measures $\Big\{\PartFunc{\sigma^2}^{-1}\d\FMeas{\sigma^2}\Big\}_{\sigma>0}$ converge weakly to the delta measure on the conjugacy class of the identity map as $\sigma\to 0$.
\end{crl}

\subsubsection{H\"{o}lder condition on $\Diff^1(\T)/\SL(2, \R)$}
It is well known that H\"{o}lder spaces play an important role in the analysis of stochastic processes. 
Here we give a description of an analogue of H\"{o}lder property for $\Diff^1(\T)/\SL(2, \R)$. 
First, we define it in terms of cross-ratio observables on $\Diff^1(\T)/\SL(2,\R)$, given by \eqref{eqDefObs_intro},
which depend only on the conjugacy class of $\phi$; see Section \ref{sect_meas_construct}.
Then we relate our definition to the classical H\"{o}lder condition for continuous functions applied to $\log \phi'$ after fixing the gauge.

\begin{defn}
For $\alpha>0$, denote by $\MyHol{\alpha}{K}$ the set of those $\phi \in \Diff^1(\T)/\SL(2, \R)$ for which
\begin{equation}\label{eqHolderSL}
\left|\obs{\phi}{s}{t}
- \frac{\pi}{\sin \big(\pi [t-s]\big)}\right|\leq  K \, \dist{s}{t}^{\alpha-1}
\end{equation}
for all $s, t\in \T$.
\end{defn}

\begin{rmrk}
The classical H\"{o}lder condition is not $\SL(2, \R)$ invariant and depends on the representative.
Condition \eqref{eqHolderSL}, on the other hand, is written in terms of well-defined observables on $\Diff^1(\T)/\SL(2, \R)$.
\end{rmrk}
\begin{rmrk}
One can also define a more local version of the H\"{o}lder condition. 
That is, we can consider $\phi$ such that \eqref{eqHolderSL} holds for  $s, t\in \T$ which are sufficiently close, say $\dist{s}{t}<\eps$ for some $\eps >0$. 
However, it follows from Theorem \ref{thrHolderEquiv} that if such local condition holds for some $\phi$, then $\phi \in \MyHol{\alpha}{K'}$ for some $K' = K'(K, \eps)$. 
\end{rmrk}

Denote 
\begin{equation}\label{eq_gauge_fix_def}
\DiffG^1(\T) = \left\{\phi\in \Diff^1(\T): \phi(0)=0,\, \phi\left(\frac{1}{3}\right)=\frac{1}{3},\, \phi\left(\frac{2}{3}\right)=\frac{2}{3}\right\},
\end{equation}
with the topology induced by $\Diff^1(\T)$. 
Notice that as a topological space it is isomorphic to $\Diff^1(\T)/\SL(2,\R)$.

The following Theorem provides a connection between our definition of $\MyHol{\alpha}{\cdot}$ and the classical notion of H\"{o}lder condition for $\log \phi'$.
Informally, we show that these definitions agree if we identify $\Diff^1(\T)/\SL(2, \R)$ with $\DiffG ^1(\T)$.
\begin{thr}\label{thrHolderEquiv}
Fix $\alpha \in (0,1)$. For any $C>0$, there exists $C'>0$ such that if $\phi \in \Diff^1(\T)$ and $\forall s, t \in \T: |\log \phi'(t)- \log \phi'(s)|<C \, \dist{s}{t}^{\alpha}$, then $\phi\in \MyHol{\alpha}{C'}$.
Furthermore, $C'$ can be chosen so that $C'\to 0$ if $C\to 0$.

A converse is also true. Fix $\eps >0$. Then for any $C'>0$ there exists $C''>0$ such that if $\phi\in \DiffG^1(\T)$ and
\begin{equation}\label{eqHolderSLThrConverse}
\left|\obs{\phi}{s}{t}
- \frac{\pi}{\sin \big(\pi [t-s]\big)}\right|\leq  C' \, \dist{s}{t}^{\alpha-1}
\end{equation}
for all $s, t\in \T$ with $\dist{s}{t}<\eps$,
then
$\forall s, t \in \T$ we have $|\log \phi'(t)-\log \phi'(s)|<C'' \dist{s}{t}^{\alpha}$.
Moreover, if $\eps >1$ (i.e., \eqref{eqHolderSLThrConverse} holds for all $s,t\in \T$), then $C''$ can be chosen so that $C''\to 0$ if $C'\to 0$.
\end{thr}
\begin{rmrk}
We can consider alternative ways to fix the gauge (i.e., choose representatives in each conjugacy class).
For example, one can fix three values of $\phi$ and/or $\phi'$ in a way which differs from \eqref{eq_gauge_fix_def}. 
This will result in a new space $\DiffGG \, ^1(\T)$ (although it will still be topologically isomorphic to $\Diff^1(\T)/\SL(2, \R)$). 
It is not hard to see that Theorem \ref{thrHolderEquiv} will still hold with the new $\DiffGG\, ^1(\T)$. 
Indeed, it is not hard to check that the usual H\"{o}lder conditions on these spaces are equivalent up to a change of the constant: for any $C>0$, there exists $C'>0$ such that if $\phi_0\in \DiffG ^1(\T)$ and $\phi_1 \in \DiffGG \, ^1(\T)$ are representatives of the same $\SL(2, \R)$ conjugacy class, then $|\log \phi_0'(t)-\log \phi_0'(s)|<C \dist{s}{t}^{\alpha}$ implies $|\log \phi_1'(t)-\log \phi_1'(s)|<C' \dist{s}{t}^{\alpha}$, and $|\log \phi_1'(t)-\log \phi_1'(s)|<C \dist{s}{t}^{\alpha}$ implies $|\log \phi_0'(t)-\log \phi_0'(s)|<C' \dist{s}{t}^{\alpha}$.
\end{rmrk}

\begin{rmrk}
The main difficulty in Theorem \ref{thrHolderEquiv} is in the converse statement. 
Notice that $\log \phi'$ being H\"{o}lder is a local condition.
However, it follows from \eqref{eqHolderSL} \emph{only} if we impose a global condition by fixing the gauge.
Without imposing such a global condition, the result fails because we can make the constant in H\"{o}lder condition for $\log \phi'$ arbitrarily large by acting with $\SL(2, \R)$.
Therefore, one cannot deduce that $\log \phi'$ is H\"{o}lder from \eqref{eqHolderSLThrConverse} by a purely local argument.
\end{rmrk}
\begin{rmrk}
The condition that $\log \phi'$ is H\"{o}lder under gauge fixing has also been studied in the context of the Weil--Petersson class of diffeomorphisms in \citep{Yilin_Holder}.
\end{rmrk}

Combining the theorem above with the Stone–Weierstrass Theorem we immediately obtain a compact embedding theorem as a corollary.
\begin{crl}\label{crlHolderClassIsCompact}
For any $\alpha \in (0,1)$ and $K>0$, the set $\MyHol{\alpha}{K}$
is compact in the topology of $\Diff^1(\T)/\SL(2, \R)$.
\end{crl}

We also show that the measures $\FMeas{\sigma^2}$ are supported on  $\MyHol{\alpha}{\cdot}$ up to an exponentially small error.
Recall that $\PartF(\sigma^2)$ is the total mass of $\FMeas{\sigma^2}$.

\begin{thr}\label{thrHolderClassExpTail}
For any $\alpha\in [1/4, 1/2)$ and any $\Lambda, N>0$ there exists $M>0$ such that for any $\sigma\in (0, \Lambda)$,
\begin{equation}\label{eqExpBoundAllPoints}
\FMeas{\sigma^2} \Big(
\MyHol{\alpha}{M}
\Big)
\\
\geq \PartF(\sigma^2)- \exp\left(-\frac{N}{\sigma^2}\right).
\end{equation}
\end{thr}

In particular, combination of Corollary \ref{crlHolderClassIsCompact} with Theorem \ref{thrHolderClassExpTail} gives exponential tightness of the family of measures $\Big\{\d \FMeas{\sigma^2}\Big\}_{\sigma\to 0}$.

\subsubsection{Minimisation of the action under constraints}
We also find the minimizers of the good rate function under certain constraints.
Fix $N\geq 2$. Let $t_1<t_2<\ldots <t_N<t_{N+1}=t_1+1$ be points on the circle $\T$.
Also, let $\left\{p_j\right\}_{j=1}^N$ be distinct points on  $\T$ lying in the anticlockwise (i.e., increasing) order, and $\left\{q_j\right\}_{j=1}^N$ be positive numbers.
We study the minimisation problem for action $\I(\cdot)$ on the following set:
\begin{equation}\label{eq_Diff_Constraints}
\DiffAppl = \left\{\phi\in\Diff^1(\T)\, \Big| \, \forall j: \, \phi(t_j)=p_j, \, \phi'(t_j)=q_j\right\}.
\end{equation}
That is, we want to find
\begin{equation}\label{eq_constr_min_problem}
\psi = \argmin_{\phi\in\DiffAppl}\, \I(\phi).
\end{equation}
Here, for convenience, we are working on the whole space $\Diff^1(\T)$ instead of the quotient $\Diff^1(\T)/\SL(2, \R)$.
Since imposing three constraints fixes representatives in each conjugacy class, this minimisation problem is equivalent to a minimisation problem on the quotient $\Diff^1(\T)/\SL(2, \R)$.
Moreover, constraints from \eqref{eq_Diff_Constraints} can be easily rewritten in terms of cross-ratio observables $\obs{\phi}{t_i}{t_j}$, so the minimisation problem \eqref{eq_constr_min_problem} can be naturally reformulated as a minimisation problem on $\Diff^1(\T)/\SL(2, \R)$ with constraints on the cross-ratio observables.

Notice that since $\I$, according to Theorem \ref{thrLDP}, is a good rate function on $\Diff^1(\T)/\SL(2, \R)$, the minimum above is indeed achieved.
Since the expression $\big(\log \phi'(t)\big)'^{\, 2}-4\pi^2\phi'^{\, 2}(\tau)$ is local, it is sufficient to optimise $\phi$ on each of the intervals $(t_j, t_{j+1})$ separately.
Below, we show how to optimise $\phi$ on $(t_1, t_2)$.

\begin{prp}\label{prp_action_min_constrained}
Let $\psi: [t_1, t_2]\to [p_1, p_2]$ be the minimizer of the functional
\begin{equation}
\phi\mapsto \dfrac{1}{2}\displaystyle\int_{t_1}^{t_2} 
\left[\Big(\log \phi'(\tau)\Big)'^{\, 2}-4\pi^2\phi'^{\, 2}(\tau)\right] \d\tau
\end{equation}
on the space $\LogSob(\T)$ under constraints 
\begin{align}
\phi(t_1) &= p_1, \quad \phi(t_2) = p_2,\\
\phi'(t_1) &= q_1, \quad \phi'(t_2) = q_2.
\end{align}
Denote
\begin{equation}
\varkappa := \frac{\pi\sqrt{q_1 q_2}(t_2-t_1)}{\sin\big(\pi(p_2-p_1)\big)}.
\end{equation}
Then for some $ a, b, c, d\in \R$ with $ad-bc= 1$ and any $\tau\in[t_1, t_2]$: 
\begin{enumerate}
\item If $\varkappa>1$
then
\begin{equation}
\tan\big(\pi\, \psi(\tau)\big) = \frac{a\tan(\lambda\tau) +b}{c \tan(\lambda\tau) +d},
\end{equation}
where $\lambda\in(0, \pi/(t_2 -t_1))$ is such that
\begin{equation}
\frac{\lambda (t_2-t_1)}{\sin\big(\lambda(t_2-t_1)\big)}
= \varkappa.
\end{equation}

\item  If $\varkappa=1$
then
\begin{equation}
\tan\big(\pi\,\psi(\tau)\big) = \frac{a \,\tau +b}{c \,\tau +d}.
\end{equation}

\item If $\varkappa<1$
then
\begin{equation}
\tan\big(\pi\,\psi(\tau)\big) = \frac{a\tanh(\lambda\tau) +b}{c \tanh(\lambda\tau) +d},
\end{equation}
where $\lambda>0$ is such that
\begin{equation}
\frac{\lambda (t_2-t_1)}{\sinh\big(\lambda(t_2-t_1)\big)}
=\varkappa.
\end{equation}
\end{enumerate}
\end{prp}

\begin{rmrk}
It is easy to see that function $x\mapsto \frac{\sin x}{x}$ is a bijection $(0, \pi)\to (0, 1)$ and $x\mapsto \frac{\sinh x}{x}$ is a bijection $(0, \infty)\to (1, \infty)$, so $\lambda$ from the proposition is always well defined.
\end{rmrk}

\begin{rmrk}
It is possible also to compute $\I(\psi)$, but we do not find the exact value to be instructive.
\end{rmrk}

\subsection{Organisation of the paper}
In Section~\ref{sect_meas_construct} we recall the construction of the Schwarzian measure (the measure corresponding to the Schwarzian Field Theory) from \citep*{BLW}
(which follows the plan proposed in \citep{BelokurovShavgulidzeExactSolutionSchwarz, BelokurovShavgulidzeCorrelationFunctionsSchwarz}) as well as some properties of cross-ratio observables from \citep{LosevCorr}.

In Section~\ref{sectProofThrActionMin} we prove the main statements about the action $\I(\cdot)$. 
That is, we prove that it is $\SL(2, \R)$ invariant (Proposition~\ref{prpActionSLInvariance}) and find its global minimiser (Theorem~\ref{thrActionMin}).

In Section \ref{sectProofThrHolderEquiv} we prove the equivalence of H\"{o}lder conditions on $\Diff^1(\T)/\SL(2, \R)$, as defined using cross-ratios and by fixing the gauge (Theorem~\ref{thrHolderEquiv}).  

Sections \ref{sectWeakLDP} and \ref{sectExpEstimates} are devoted to the proof of the large deviations principle for $\left\{\d\FMeas{\sigma^2}\right\}_{\sigma>0}$ (Theorem~\ref{thrLDP}).
It is well known that if an exponentially tight family of measures satisfies a weak large deviations principle with a rate function $I(\cdot)$, then $I(\cdot)$ is a good rate function, and the family of measures in question satisfies large deviations principle with a good rate function $I(\cdot)$ (see, e.g., \citep[Lemma 1.2.18]{bookLDP}). 

In Section~\ref{sectWeakLDP} we prove that $\I(\cdot)$ is a rate function and that the family of measures $\left\{\d\FMeas{\sigma^2}\right\}_{\sigma>0}$ satisfies weak large deviations principle with the rate function $\I(\cdot)$.

In Section~\ref{sectExpEstimates} we prove exponential estimates and deduce that the measures $\left\{\d\FMeas{\sigma^2}\right\}_{\sigma >0}$ concentrate on $\MyHol{\alpha}{\cdot}$ up to exponentially small sets when $\alpha<1/2$ (Theorem~\ref{thrHolderClassExpTail}).
In particular, this proves exponential tightness of $\left\{\d\FMeas{\sigma^2}\right\}_{\sigma>0}$ (as a combination of Theorem~\ref{thrHolderClassExpTail} and Corollary~\ref{crlHolderClassIsCompact}). 
This finishes the proof of Theorem~\ref{thrLDP}.

\subsection{Preliminaries and notations}
\label{sectNotation}

Throughout the paper we will be using the following notations:

\begin{enumerate}

\item The unit circle is denoted by $\T=[0,1]/\{0\sim 1\}$, the nonnegative real numbers are denoted by $\R_+ = [0, \infty)$, and for the open disk in the complex plane of radius $r$ we use $\D_r = \left\{z\in \Compl: |z|<r\right\}$.
Moreover, for $s, t\in \T$, we write $t-s$ for the length of the interval going from $s$ to $t$ in positive direction.
In particular, $t-s\in [0,1)$.
We also use $\dist{s}{t} = \min\big\{t-s, 1-(t-s)\big\}$ for the distance on the circle. 

\item We use $\Diff^k(\T)$ for the set of oriented $C^k$-diffeomorphisms of $\T$, that is, $\phi\in \Diff^k(\T)$
can be identified with a $k$-times continuously differentiable function $\phi\colon\R\to \R$ satisfying
$\phi(\tau+1)=\phi(\tau)+1$ and $\phi'(\tau)>0$ for all $\tau \in \R$. Note that $\Diff^k(\T)$ is not a linear space.
The topology on $\Diff^k(\T)$ is inherited from the natural topology on $C^k(\T)$.
It turns $\Diff^{k}(\T)$ into a Polish (separable completely metrisable) space as well as a topological group.

\item We use $\SL(2, \R)$ to denote the group of conformal isomorphisms of the unit disk restricted to the boundary, which is identified with the unit circle $\T$.
Throughout the paper we consider the action of $\SL(2, \R)$ on $\Diff^1(\T)$ given by post-compositions.

\item We will abuse the notation and for $\phi\in \Diff^1(\T)$ denote its conjugacy class in $\Diff^1(\T)/\SL(2,\R)$ by $\phi$ as well.

\item Throughout the paper we will encounter expressions of the form $f\big(\arccosh[z]\big)$, for various even analytic functions $f$.
Even though $\arccosh[z]$ is not analytic at $z=0$, the composition $f\big(\arccosh[z]\big)$ still defines an analytic function around $z=0$.
More precisely, we identify $f\big(\arccosh[z]\big)$ with $\widetilde{f}\big(\arccosh^2[z]\big)$, where $\widetilde{f}$ is an analytic function such that $\widetilde{f}(\omega) = f(\sqrt{\omega})$, and  $\arccosh^2[z]$ is the analytic function described in statement below (see \citep{LosevCorr} for details).
\begin{stm}[\citep{LosevCorr}]\label{stmArccoshDef}
Function $\arccosh^2(z)$ can be analytically continued from $z\in [1, \infty)$ to $z \in \Dom := \left\{z\in\Compl \, \Big| \, \Real z\geq -1 \right\}$. 
Moreover, in this continuation $\arccosh^2(z) = -\arccos^2(z)$ for $z \in (-1, 1)$.
\end{stm}
\end{enumerate}

\section{Measure construction and observables}\label{sect_meas_construct}
\subsection{Schwarzian measure}
In this subsection we recall the rigorous construction and main properties of the measure corresponding to Schwarzian Field Theory from \citep*{BLW}, which are based on the plan from \citep{BelokurovShavgulidzeExactSolutionSchwarz, BelokurovShavgulidzeCorrelationFunctionsSchwarz}. 
We refer the reader to that paper for proofs and further details.
\medskip

The construction of the Schwarzian measure is based on the appropriate reparametrisation of  unnormalised version of the Brownian bridge measure.
This is a finite measure on $\Cfree[0,T] = \left\{f\in C[0,T]\, | \, f(0)=0\right\}$ formally corresponding to
\begin{equation} \label{e:BB-formaldensity}
\d\WS{\sigma^2}{a}{T}(\xi) = \exp\left\{-\frac{1}{2\sigma^2}\int_{0}^{T}\xi'^{\, 2} (t)\d t\right\} \delta\big(\xi(0)\big) \delta\big(\xi(T)-a\big)\prod_{\tau \in (0,T)}\d\xi(\tau).
\end{equation}

\begin{defn} \label{defn:BB}
  The unnormalised Brownian bridge measure with variance $\sigma^2 > 0$ is a finite Borel measure $\d\WS{\sigma^2}{a}{T}$ on $\Cfree[0, T]$ such that
  \begin{equation}\label{eq:25}
    \sqrt{2\pi T}\sigma \, \exp\left\{\frac{a^2}{2 T\sigma^2}\right\}\d\WS{\sigma^2}{a}{T}(\xi)
  \end{equation}
  is the distribution of a Brownian bridge $\big(\xi(t)\big)_{t\in[0,T]}$ with variance $\sigma^2$ and $\xi(0) = 0$, $\xi(T) = a$.
\end{defn}

In order to define the Schwarzian measure $\FMeas{\sigma^2}$,
we first need to define a finite measure $\mu_{\sigma^2}$ on $\Diff^1(\T)$ which is similar to what is known as the Malliavin--Shavgulidze measure;
see \cite[Section~11.5]{BogachevMalliavin}. %
Formally, this measure corresponds to
\begin{equation} \label{e:MeasureMuFormal}
\d\mu_{\sigma^2}(\phi) 
= \exp\left\{-\frac{1}{2\sigma^2}\int_{0}^{1}\left(\frac{\phi''(\tau)}{\phi'(\tau)}\right)^2\d\tau\right\} \prod_{\tau \in [0,1)}\frac{\d\phi(\tau)}{\phi'(\tau)}.
\end{equation}

We can make sense of this measure by defining it as a push-forward of an unnormalised Brownian bridge on $[0,1]$ with respect to a suitable change of variables.
We \emph{define} $\mu_{\sigma^2}$ by
\begin{equation}\label{defMeasureMu}
\d\mu_{\sigma^2}(\phi) \coloneqq 
\d \WS{\sigma^2}{0}{1}(\xi)
\otimes \d\Theta, \qquad \text{with } \phi(t) = \Theta + \A_{\xi}(t)\enspace (\mathrm{mod}\, 1), \text{ for } \Theta\in [0, 1),
\end{equation}
where $\d\Theta$ is the Lebesgue measure on $[0,1)$ and
\begin{equation} \label{defP}
\A(\xi)(t) \coloneqq \A_{\xi}(t) 
\coloneqq \frac{\int_{0}^t e^{\xi(\tau)}\d\tau}{\int_{0}^1 e^{\xi(\tau)}\d\tau}.
\end{equation}
The variable $\Theta$ corresponds to the value of $\phi(0)$.
Note that the map $\xi \mapsto \A(\xi)$ is a bijection between $\Cfree[0, 1]$
and $\Diff^1 [0,1]$ with inverse map
\begin{align} 
\A^{-1}: \Diff^1[0,1] &\to \Cfree[0, 1]\\
\varphi &\mapsto \log \varphi'(\cdot) - \log \varphi'(0).
\end{align}
Note that, with our choice of $C^1$ topology on $\Diff^1[0,1]$ and the usual supremum norm topology on $\Cfree[0, 1]$, we get that both $\A$ and $\A^{-1}$ are continuous.

In view of \eqref{eq:1} and \eqref{e:MeasureMuFormal},
the unquotiented Schwarzian measure is constructed as
\begin{equation} \label{defMeasureMeas}
\d\FMeasSL{\sigma^2}(\phi) = \exp\left\{ \frac{2\pi^2}{\sigma^2}\int_{0}^{1} \phi'^{\, 2}(\tau)\d\tau\right\}\d \mu_{\sigma^2}(\phi).
\end{equation}
Since $\mu_{\sigma^2}$ is supported on $\Diff^1(\T)$, this defines a Borel measure on $\Diff^1(\T)$. 
This is the unique (up to a multiplicative constant) measure satisfying a natural change of variables formula.
In particular, it is invariant under the action of $\SL(2,\R)$.
\begin{prp}\label{lemmaMeasureSLInv}
The measure $\FMeasSL{\sigma^2}$ is invariant under post-composition by elements of $\SL(2, \R)$. 
In other words, for any $\psi \in \SL(2, \R)$ and Borel $A\subset \Diff^1(\T)$ we have
\begin{equation}
  \FMeasSL{\sigma^2}\big(\psi\circ A\big) = \FMeasSL{\sigma^2}\big(A\big),
\end{equation}
where  $\psi\circ A \coloneqq \left\{\psi\circ \phi\, \big | \, \phi \in A \right\}$.
\end{prp}

In particular, we remark that  $\FMeasSL{\sigma^2}$ is an infinite measure since $\SL(2,\R)$ has infinite Haar measure.
The Schwarzian measure on $\Diff^1(\T)/\SL(2,\R)$ with formal density \eqref{eq:1} is defined
as the quotient of $\FMeasSL{\sigma^2}$ by $\SL(2,\R)$; see Proposition~\ref{propMeasureFactor} and Definition~\ref{defn_schwarzian} below.

\begin{prp} \label{propMeasureFactor}
  There exists a unique Borel measure $\FMeas{\sigma^2}$ on $\Diff^1(\T)/\SL(2,\R)$ such that, for any nonnegative continuous functional $F\colon \Diff^{1}(\T) \to \R$,
  \begin{equation}
    \label{e:measurefactor}
    \int\limits_{\mathclap{\Diff^1(\T)}}\d\FMeasSL{\sigma^2}(\phi)\, F(\phi)
    \hspace*{1em}
    =
    \hspace*{2.5em}
    \int\limits_{\mathclap{\Diff^1(\T)/\SL(2,\R)}}\d\FMeas{\sigma^2}([\phi])
    \hspace*{1em}
    \int\limits_{\mathclap{\SL(2,\R)}} \d\nu_H(\psi)\, F(\psi\circ\phi),
  \end{equation}
  where the right-hand side is well-defined, since the second integral only depends on $[\phi]$.
\end{prp}
\begin{defn} \label{defn_schwarzian}
  The Schwarzian measure is given by $\FMeas{\sigma^2}$.
\end{defn}
This measure is finite, and moreover, its total mass can be computed explicitly.

\begin{prp}\label{prpMainPartFunct}
The partition function (i.e., total mass) of $\FMeas{\sigma^2}$ is given by
\begin{equation}
\PartF(\sigma^2) =  
    \left(\frac{2\pi}{\sigma^2}\right)^{3/2} \exp\left(\frac{2\pi^2}{\sigma^2}\right)
    =\int_0^{\infty} e^{-{\sigma^2k^2}/{2}} \sinh(2\pi k) \, 2 k \d k.
\end{equation}
\end{prp}

\subsection{Observables}
Here we recall the main properties of cross-ratio observables obtained in \citep{LosevCorr}.

The natural observables on the space $\Diff^1(\T)/\SL(2, \R)$ are given by cross-ratios \eqref{eqDefObs_intro}, since they depend only on the conjugacy class under $\SL(2, \R)$ action.
\begin{prp}\label{prpObsSLInvar}
Observables $\obs{\phi}{s}{t}$ are invariant under post-compositions by elements of M\"{o}bius transformations. 
In other words, if $\psi\in \SL(2, \R)$, then
\begin{equation}
\obs{\psi\circ\phi}{s}{t} = \obs{\phi}{s}{t}.
\end{equation}
In particular, they induce well-defined observables on $\Diff^1(\T)/\SL(2, \R)$, which we, slightly abusing the notation, also denote by $ \obs{\phi}{\cdot}{\cdot}$.
\end{prp}
Importantly, it is possible to compute correlation functions of cross-ratios \citep{LosevCorr}.
Here we will only need to know their moments.
\begin{defn}
For all $l, k, w \in \R$, we define $\Gamma(l\pm ik\pm iw)$ as
\begin{equation}
\Gamma(l\pm ik\pm iw) := 
\Gamma(l + ik + iw)  \Gamma(l + ik - iw)  \Gamma(l - ik + iw)  \Gamma(l - ik - iw) .
\end{equation}
\end{defn}

\begin{prp}\label{prpMainObsMoments}
Moments of cross-ratio observables for positive integers $l$ are given by
\begin{multline}
\int \obs{\phi}{s}{t}^{l}\d\FMeas{\sigma^2}(\CnjCl{\phi}) 
= 
\int_{\R_+^{2}} 
\frac{\Gamma\big(\frac{l}{2} \pm i k_1 \pm i k_2\big)}{2\pi^2\, \Gamma(l)}\cdot\left(\frac{\sigma^2}{2}\right)^l
\\
\times
\exp\left(-\frac{(t-s)\sigma^2}{2}\cdot k_1^2 -\frac{\big(1-(t-s)\big)\sigma^2}{2}\cdot k_2^2\right)  
\sinh(2\pi k_1)\, 2 k_1 \sinh(2\pi k_2)\, 2 k_2 \d k_1 \d k_2.
\end{multline}
\end{prp}
Using this, it is possible to show that cross-ratios have finite exponential moments.
\begin{prp}\label{prpExpMoment}
For any $\sigma>0$ and any $s\neq t \in \T$, 
\begin{equation}
\int  \exp\left\{ \frac{8}{\sigma^2}\, \obs{\phi}{s}{t} \right\} \d\FMeas{\sigma^2}(\CnjCl{\phi}) < \infty.
\end{equation}
\end{prp}
In this work we give a sharper estimate for the exponential moments in the case when $s$ and $t$ are close; see Proposition \ref{prpBoundExpMoments}.
The main tool for estimating exponential moments is the following identity.
\begin{prp}\label{prpKeyArccoshExpansion}
For any $\beta, k \in \R$ and any $z\in \D_2$, we have
\begin{multline}
\cos\Big(2k\cdot \arccosh\left[\cosh(\beta/2)-z\right]\Big)
=\\
\cos(k\beta)+
2k\sinh(2\pi k)
\int_0^{\infty}
\sum_{l=1}^{\infty} \frac{\Gamma\Big(\frac{l}{2}\pm i k \pm i w\Big)}{2\pi^2 \Gamma(l)}\cdot \frac{(2 z)^l}{l!} %
\cos(w\beta)\d w,
\end{multline}
where the right-hand side converges absolutely.
\end{prp}
\begin{rmrk}
The right-hand side above can also be written as 
\begin{equation}
2k\sinh(2\pi k)
\int_0^{\infty}
\sum_{l=0}^{\infty} \frac{\Gamma\Big(\frac{l}{2}\pm i k \pm i w\Big)}{2\pi^2 \Gamma(l)}\cdot \frac{(2 z)^l}{l!} %
\cos(w\beta)\d w,
\end{equation}
if we interpret $l=0$ term as delta function $\delta(\omega-k)$.
\end{rmrk}

\section{Action functional $\I(\cdot)$}
\label{sectProofThrActionMin}
In this section we prove statements concerning the action functional $\I(\cdot)$. 
That is, here we show that it is $\SL(2, \R)$ invariant (Proposition~\ref{prpActionSLInvariance}) and find its global minimisers (Theorem~\ref{thrActionMin}).

\begin{proof}[Proof of Proposition \ref{prpActionSLInvariance}]
For $\phi\in \Diff^{\infty}(\T)$ we have 
\begin{equation}
\I(\phi) 
= \int_{0}^{1} 
\Big[-\Schw \left(\phi, \tau\right) -2\pi^2\phi'^{\, 2}(\tau)\Big] \d\tau
= -\int_{0}^{1} 
\Schw \big(\tan(\pi\, \phi), \tau\big)\d\tau.
\end{equation}
Also, for any $\psi\in \SL(2, \R)$ and any $x\in \R$,
\begin{equation}\label{eqActionSLInvFormulaSLUpperHalfPlane}
\tan\left(\pi \, \psi\left(\tfrac{1}{\pi}\arctan(x)\right)\right) = \frac{ax+b}{cx+d},\qquad 
\text{for some } a, b, c, d\in \R, \text{ with } ad-bc = 1.
\end{equation}
This follows from the fact that the map $\phi\in \T\mapsto \tan(\pi\phi)$ is the restriction to the boundary of the conformal map $z\mapsto i\frac{1+z}{1-z}$ from the unit disk in $\Compl$ to the upper half plane, whose group of conformal isomorphisms coincides with the group of fractional linear transformations.
It is well known that the Schwarzian derivative $\Schw$ is invariant under transformations of the form \eqref{eqActionSLInvFormulaSLUpperHalfPlane}. Therefore, by Remark \ref{rmrk_good_rate_for_smooth}, for any $\phi \in \Diff^{\infty}(\T)$,
\begin{equation}
\I(\psi\circ \phi) = \I(\phi).
\end{equation}
It is easy to see that  $\Diff^{\infty}(\T)$ is dense in $\LogSob(\T)$ and that the map $\phi\mapsto \psi\circ \phi$ is continuous as a map $\LogSob(\T)\to\LogSob(\T)$. 
Thus, using continuity of $\I(\cdot)$ on $\LogSob(\T)$, we deduce that $\I(\phi)$ is $\SL(2, \R)$ invariant.
\end{proof}

\begin{proof}[Proof of Theorem \ref{thrActionMin}]
First, we prove the inequality.

Observe that it is sufficient to prove the inequality for $\Diff^{\infty}(\T)$, since $\Diff^{\infty}(\T)$ is dense in $\LogSob(\T)$.

Notice that
\begin{equation}\label{eqLemmaActionBoundActionFormula}
\frac{1}{2}\int_{0}^{1} \left[\left(\log \phi'(\tau)\right)'^{\,2}-4\pi^2\phi'^{\, 2}(\tau)\right] \d\tau 
=
-\int_{0}^{1} \Big[\Schw\big(\phi, \tau\big) + 2\pi^2\phi'^{\, 2}(\tau)\Big]\d\tau.
\end{equation}
Using the fact that, for any $b\in \R$,  $\Schw\big(\cot\big(\pi(\tau-b)\big), \tau\big) = 2\pi^2$ and the well-known composition rule
\begin{equation}
\Schw_{f\circ g} = \big(\Schw_f \circ g \big)\cdot (g')^2 + \Schw_g, \qquad \forall f, g \in C^3,
\end{equation}
it is easy to see that, for any $a\in \T$,
\begin{equation}\label{eqLemmaActionBoundSchDerTan}
\Schw(\phi, \tau) + 2\pi^2\phi'^{\, 2}(\tau)= 
\Schw \Big(-\cot\Big(\pi\big[ \phi(\tau)-\phi(a)\big]\Big), \tau\Big).
\end{equation}
Denote
\begin{align}
f_a(t) &= -\cot\Big(\pi\big[\phi(t+a)-\phi(a)\big]\Big);\\
q(\tau) &=  -\frac{1}{2}\Big(\Schw(\phi, \tau ) + 2\pi^2\phi'^{\, 2}(\tau)\Big).\label{eqLemmaActionBoundSchNotation}
\end{align}
Note that $f'_a(t)>0$.
It follows from \eqref{eqLemmaActionBoundSchDerTan} that
\begin{equation}
q(t+a) = -\frac{1}{2}\Schw (f_a, t).
\end{equation}
It is easy to see that if 
$v(t) = f_a'(t)^{-1/2}$,
then
\begin{equation}
v''(t)-q(a+t)v(t) = 0.
\end{equation}
Observe that $v(\tau)$ is a solution for the Sturm--Liouville eigenvector problem
\begin{equation}\label{eqLemmaActionBoundSturmLioville}
-F''(t)+q(t+a)F(t) = \lambda F(t), \qquad F(0)=F(1)=0,
\end{equation}
with $\lambda=0$. Moreover, since $v(t)\neq 0$ for $t\in(0, 1)$, we deduce that $v$ is the lowest eigenvector for the Sturm--Liouville eigenvector problem \eqref{eqLemmaActionBoundSturmLioville} (see, e.g., \citep[Chapter XI, Theorem 4.1]{ODEHartman}). 
Therefore, by the Rayleigh principle (see, e.g., \citep[Chapter VII]{bookSturmLiouville}), 
\begin{equation}
\inf_{\substack{u\in C^{\infty}, u\not\equiv 0\\
u(0)=u(1)=0}} \frac{\int_0^1 \left(u'^{\, 2}(t)+q(t+a)u^2(t)\right)\d t}{\int_0^1 u^2(t)\d t} =0.
\end{equation}
In particular, for $u(t) = \sin(\pi t)$ (which is a solution to \eqref{eqLemmaActionBoundSturmLioville} with $q=-\pi^2$ and $\lambda=0$), we obtain
\begin{equation}\label{eqLemmaActionBoundWithA}
\frac{\pi^2}{2} + \int_0^1 q(t+a)\sin^2(\pi t)\d t\geq 0.
\end{equation}
Integrating over $a\in[0, 1)$, we get
\begin{equation}\label{eqLemmaActionBoundFinalBound}
\int_0^1 q(t)\geq -\pi^2.
\end{equation}
It follows from \eqref{eqLemmaActionBoundActionFormula}, \eqref{eqLemmaActionBoundSchNotation}, and \eqref{eqLemmaActionBoundFinalBound} that the inequality \eqref{eqLemmaActionBoundResult} holds.

\medskip

Second, we prove that the equality holds only for M\"{o}bius transformations. 
Let $\phi\in\LogSob(\T)$ be such that $\I(\phi) = -2\pi^2$.
We start by showing that $\phi\in \Diff^{\infty}(\T)$.
Denote $f(t) = \log \phi'(t) - \log \phi'(0) \in C_0[0, 1]\cap \Sob[0,1]$. 
We get that
\begin{equation}
\I(\phi) = \frac{1}{2}\int_0^1 \left[f'^{\, 2}(\tau) - 4\pi^2\frac{e^{2 f(\tau)}}{\int_0^1 e^{f(t)}\d t}\right]\d \tau.
\end{equation}
The right-hand side is invariant under additions of constants to $f$, so we deduce that it obtains its maximum over $\Sob(\T)$ at $f$. 
Therefore, its functional derivative is equal to $0$.
Thus, for any $g\in C^{\infty}(\T)$,
\begin{equation}
\int_0^1f'(\tau)g'(\tau)\d\tau 
= 4\pi^2 \frac{\int_0^1 e^{2f(\tau)}g(\tau)\d\tau}{\int_0^1 e^{f(\tau)}\d\tau} 
- 2\pi^2 \left(\int_0^1 e^{f(\tau)}g(\tau)\d\tau \right)\cdot
\frac{\int_0^1 e^{2f(\tau)}\d\tau }{\left(\int_0^1 e^{f(\tau)}\d\tau\right)^2}.
\end{equation}
So, we get that $f''$ exists as a weak derivative, and
\begin{equation}
f''(t) =- 4\pi^2  \frac{e^{2f(t)}}{\int_0^1 e^{f(\tau)}\d\tau}
+ 2\pi^2\, e^{f(t)}
\frac{\int_0^1 e^{2f(\tau)}\d\tau }{\left(\int_0^1 e^{f(\tau)}\d\tau\right)^2}.
\end{equation}
Since the right-hand side is a continuous function, we deduce that $f$ is twice continuously differentiable. 
Differentiating the equation above, we get that $f\in C^{\infty}(\T)$.
Therefore, $\phi\in \Diff^{\infty}(\T)$.

Now, since $\phi\in \Diff^{\infty}(\T)$, we can repeat the argument from the first part of the proof.
The equality in \eqref{eqLemmaActionBoundFinalBound} is obtained
only if for almost all $a\in \T$,  $\sin(\pi t)$ is an eigenfunction of \eqref{eqLemmaActionBoundSturmLioville} with $\lambda = 0$.
Thus, for all $a\in \T$, there exists $C_a\neq 0$ such that $v_a(t) = C_a\sin(\pi t)$, which means that, for some  $C'_a$, 
\begin{equation}
f_a(t) = C_a \cot \big(\pi \,t\big)+C'_a,
\end{equation}
which holds only when $\phi$ is a M\"{o}bius transformations of the unit disc.  
\end{proof}

\section{Proof of Theorem~\ref{thrHolderEquiv}}
\label{sectProofThrHolderEquiv}
For reader's convenience we recall the statement of Theorem~\ref{thrHolderEquiv}.
\begin{customthm}{\ref{thrHolderEquiv}}
Fix $\alpha \in (0,1)$. For any $C>0$, there exists $C'>0$ such that if $\phi \in \Diff^1(\T)$ and $\forall s, t \in \T: |\log \phi'(t)- \log \phi'(s)|<C \, \dist{s}{t}^{\alpha}$, then $\phi\in \MyHol{\alpha}{C'}$.
Furthermore, $C'$ can be chosen so that $C'\to 0$ if $C\to 0$.

A converse is also true. Fix $\eps >0$. Then for any $C'>0$, there exists $C''>0$ such that if $\phi\in \DiffG^1(\T)$ and
\begin{equation}\label{eqHolderSLThrConverse2}
\left|\obs{\phi}{s}{t}
- \frac{\pi}{\sin \big(\pi [t-s]\big)}\right|\leq  C' \, \dist{s}{t}^{\alpha-1}
\end{equation}
for all $s, t\in \T$ with $\dist{s}{t}<\eps$,
then
$\forall s, t \in \T$ we have $|\log \phi'(t)-\log \phi'(s)|<C'' \dist{s}{t}^{\alpha}$.
Moreover, if $\eps >1$ (i.e., \eqref{eqHolderSLThrConverse2} holds for all $s,t\in \T$) then $C''$ can be chosen so that $C''\to 0$ if $C'\to 0$.
\end{customthm}

\begin{proof}
\textbf{Direct statement.} 

Notice that because of rotational invariance it is sufficient to prove \eqref{eqHolderSL} only when $t-s\leq 1/2$.
In this proof $C_1, C_2, \ldots$ denote constants that depend only on $C$ (in particular, they do not depend on $\phi$), and they all go to $0$ as $C\to 0$ (this can be verified step by step: for any $j$ if $C_j\to 0$, then $C_{j+1}\to 0$).
Observe that
\begin{equation}\label{eqThrHolderEquivDirectPhiRatioEst}
\left|\frac{\phi'(t)}{\phi'(s)}-1\right|
\leq C_1 
\,\dist{s}{t}^{\alpha}
\end{equation} 
and, in particular, 
\begin{equation}\label{eqThrHolderDirectPhiDerivAbsBound}
\forall \tau\in \T: (1+C_2)^{-1}<\phi'(\tau)<1+C_2.
\end{equation}
Integrating the inequality above, we obtain
\begin{equation} \label{eqThrHolderEquivDirectPhiEst}
\left|\phi(t)-\phi(s) - (t-s)\phi'(s)\right| \leq C_3 
\,\dist{s}{t}^{\alpha+1}
\end{equation}
Therefore, it is not hard to show that
\begin{equation}\label{eqThrHolderEquivDirectSinSqEst}
\left|\sin^2\Big(\pi \big[\phi(t) - \phi(s)\big] \Big) - \phi'^{\, 2}(s) \sin^2\big(\pi [t- s]\big)\right|
\leq C_4 
\,\dist{s}{t}^{\alpha+2}
\end{equation}
Indeed, from \eqref{eqThrHolderEquivDirectPhiEst}
\begin{equation}
\left|\sin\Big(\pi \big[\phi(t) - \phi(s) \big]\Big) -  \sin\Big(\phi'(s) \, \pi [t-s] \Big)\right| \leq  C_5 
\,\dist{s}{t}^{\alpha+1}
\end{equation}
and from \eqref{eqThrHolderDirectPhiDerivAbsBound} and Statement \ref{stmAppendixIneqSin}, we get
\begin{equation}
\left|\sin\Big(\phi'(s)\, \pi [t- s]\Big) - \phi'(s) \sin\big(\pi [t- s]\big) \right| \leq C_6 
\,\dist{s}{t}^{2}
\end{equation}
which implies \eqref{eqThrHolderEquivDirectSinSqEst}.

Also, from \eqref{eqThrHolderEquivDirectPhiRatioEst} and \eqref{eqThrHolderDirectPhiDerivAbsBound}, we get
\begin{equation}\label{eqThrHolderEquivDirectPhiProdEst}
\left|\phi'(t)\phi'(s) - \phi'^{\, 2}(s)\right| \leq C_7 
\,\dist{s}{t}^{\alpha}
\end{equation}

Combining \eqref{eqThrHolderEquivDirectSinSqEst} and \eqref{eqThrHolderEquivDirectPhiProdEst} we deduce that if $\dist{s}{t}<\frac{1}{100} (1+C_2)^{-2/\alpha}\, C_4^{-1/\alpha}$, then
\begin{equation}
\left|
\frac{\phi'(s)\phi'(t)}{\sin^2\Big(\pi \big[\phi(t)-\phi(s)\big]\Big)} 
- \frac{1}{\sin^2 \big(\pi [t-s]\big)}\right|\leq  C_8 
\,\dist{s}{t}^{\alpha-2}
\end{equation}

If $C$ is small enough, we get that $\frac{1}{100} (1+C_2)^{-2/\alpha}\, C_4^{-1/\alpha}>1$, and thus, the inequality above holds for all $s, t\in \T$.

Otherwise, it follows from \eqref{eqThrHolderEquivDirectSinSqEst} and monotonicity of $\phi$ that $C_9^{-1}<\sin^2\Big(\pi \big[\phi(t)-\phi(s)\big]\Big)<C_9$ whenever $\dist{s}{t}\geq \frac{1}{1000} (1+C_2)^{-2/\alpha}\, C_4^{-1/\alpha}$.
Therefore, by compactness argument and \eqref{eqThrHolderDirectPhiDerivAbsBound} we get that
\begin{equation}
\left|
\frac{\phi'(s)\phi'(t)}{\sin^2\Big(\pi \big[\phi(t)-\phi(s)\big]\Big)} 
- \frac{1}{\sin^2 \big(\pi [t-s]\big)}\right|\leq  C_{10} 
\,\dist{s}{t}^{\alpha-2}
\end{equation}
for all $s, t\in \T$ with $\dist{s}{t}\geq \frac{1}{1000} (1+C_2)^{-2/\alpha}\, C_4^{-1/\alpha}$, which finishes the proof.
\medskip

\textbf{Converse statement.}

We let $\delta \in (0, 1/100)$ be a sufficiently small number to be chosen later.
We consider two cases: in the first case $\eps>1$ and $C'<\delta$, in the second case either $\eps \leq 1$ or $C'\geq\delta$.
In this proof $C'_1, C'_2, \ldots$ denote constants, that depend only on $C'$. 
Moreover, in the first case, they all go to $0$ as $C'\to 0$.

\medskip

\textbf{Case 1: } $\eps>1$ and $C'<\delta$.

For every $s\in \T$ we let $\Psi_s \in \SL(2, \R)$ be a M\"{o}bius transformation of a unit disc such that 
\begin{equation}\label{eqThrHolderConverseCaseOneGauge}
\Psi_s\circ \phi (s) = 0, \qquad \big(\Psi_s\circ \phi\big)' (s) = 1, \qquad \Psi_s\circ \phi \left(s+\tfrac{1}{2}\right) = \tfrac{1}{2}.
\end{equation} 

We start by proving H\"{o}lder estimate for $\Psi_s\circ \phi$ when one of the points in question is equal to $s$. 
That is, we are going to prove that for all $s, t\in \T$,
\begin{equation}\label{eqThrHolderConverseCaseOneGaugeHolder}
\Big|\log \big[ \big(\Psi_s\circ \phi\big)'(t)\big] -  \log \big[ \big(\Psi_s\circ \phi\big)'(s)\big] \Big| \leq C'_1   \,\dist{s}{t}^{\alpha}
\end{equation}

Notice, that by symmetry it is sufficient to prove \eqref{eqThrHolderConverseCaseOneGaugeHolder} for $t\in [s, s+\frac{1}{2}]$.
Using \eqref{eqHolderSLThrConverse2} and \eqref{eqThrHolderConverseCaseOneGauge}, we get that for all $t\in[s, s+\frac{1}{2}]$
\begin{equation}\label{eqThrHolderConverseCaseOneGaugeIneq}
\left|\frac{\big(\Psi_s\circ \phi\big)'(t)}{\sin^2\Big(\pi \, \Psi_s\circ \phi(t)\Big)} -\frac{1}{\sin^2\big(\pi [t- s]\big)}\right| 
\leq C' 
\dist{s}{t}^{\alpha-2}.
\end{equation}
Denote $F_s(\tau) = \tan\big( \pi \, \Psi_s\circ \phi (s+\tau)\big)$. 
Then for all $\tau \in(0, 1/2)$, we have
\begin{equation}\label{eqThrHolderConverseCaseOneDiffBound}
\left| \frac{F'_s(\tau)}{F_s^2(\tau)}-\frac{\pi}{\sin^2(\pi \tau)}\right|\leq C_2'\tau^{\alpha-2}.
\end{equation}
Integrating this from $x$ to $\frac{1}{2}$, we obtain
\begin{equation}
\left|\frac{1}{F_s(x)} - \cot(\pi x)\right| \leq C'_3 x^{\alpha-1}, \qquad \text{for } x\in (0, 1/2) .
\end{equation}
Therefore,
\begin{equation}
\left| \frac{1}{F_s^2(x)} - \cot^2(\pi x)\right|  \leq C'_4 x^{\alpha-2}.
\end{equation}
Notice that
\begin{equation}
\frac{1}{\sin^2\Big(\pi\, \Psi_s\circ \phi (s+\tau)\Big)} = 1+\frac{1}{F_s^2(\tau)}.
\end{equation}
Hence,
\begin{equation}
\left|\frac{1}{\sin^2\Big(\pi\, \Psi_s\circ \phi (s+\tau)\Big)} -\frac{1}{\sin^2(\pi \tau)}\right| \leq C_5' \tau^{\alpha-2}.
\end{equation}
For sufficiently small $\delta$, we have $C'_5<1/100$, and hence,
\begin{equation}
\sin^2\Big(\pi\, \Psi_s\circ \phi (s+\tau)\Big) \leq C'_6\tau^2
\end{equation}
and
\begin{equation}
\left|1- \frac{\sin^2\Big(\pi\, \Psi_s\circ \phi (s+\tau)\Big)}{\sin^2(\pi \tau)}\right| \leq C_7' \tau^{\alpha}.
\end{equation}
Combining this with \eqref{eqThrHolderConverseCaseOneGaugeIneq}, we get
\begin{equation}
\left|\Big(\Psi_s\circ \phi\Big)'(s+\tau) - 1\right|\leq C_8' \tau^{\alpha},
\end{equation}
which implies \eqref{eqThrHolderConverseCaseOneGaugeHolder}.

The inequality above also implies that, for all $0\leq u\leq v\leq 1$,
\begin{equation}\label{eqThrHolderConverseCaseOneGaugeFunctionEst}
\left|\Psi_s \circ \phi(v)-\Psi_s \circ \phi(u)-(v-u)\right|\leq C'_9|v-u|.
\end{equation}
In particular,
\begin{align}
\left| \Psi_s \circ \phi \left(\tfrac{1}{3}\right) - \Psi_s\circ \phi\left(0\right)  - \tfrac{1}{3} \right| &\leq C'_{10} ,\\
\left|\Psi_s\circ \phi\left(\tfrac{2}{3}\right) -  \Psi_s\circ \phi\left(\tfrac{1}{3}\right)  - \tfrac{1}{3}\right| &\leq C'_{10}, \\
\left|\Psi_s\circ \phi\left(1\right) - \Psi_s\circ \phi\left(\tfrac{2}{3}\right)  - \tfrac{1}{3}\right| &\leq C'_{10}.
\end{align}
Using Statement \ref{stmAppendixSLDerivBounds}, we deduce that, for all $u, v\in [0,1]$,
\begin{equation}\label{eqThrHolderConverseCaseOneInverseContinuity}
\left|\log \left[\Psi^{(-1)}_s\right]'(u)-\log \left[\Psi^{(-1)}_s\right]'(v)\right|
\leq C'_{11} \dist{u}{v}
\end{equation}
where $\Psi^{(-1)}_s$ is the inverse to $\Psi_s$.
Writing
\begin{equation}
\log \phi'(t) = \log \Big(\left[\Psi^{(-1)}_s\right]'\big(\Psi_s\circ \phi(t)\big)\Big) +\log \Big( \big(\Psi_s\circ \phi\big)'(t)\Big),
\end{equation}
and combining this with \eqref{eqThrHolderConverseCaseOneGaugeHolder}, \eqref{eqThrHolderConverseCaseOneGaugeFunctionEst}, and \eqref{eqThrHolderConverseCaseOneInverseContinuity}, we obtain the desired.

\medskip

\textbf{Case 2: } Either $\eps<1$ or $C'\geq \delta$.

This Case is similar to the Case 1.  Essentially, we cut the whole circle into many small subintervals, and then we apply an argument from Case 1 to each subinterval, proving a version of "local" H\"{o}lder estimate. 
Then we show that the global estimates follow from the local estimates.

First, we divide the circle into $N$ equal intervals with $N =10 (1+C')^{ 1/\alpha}\, \eps ^{-1/{\alpha}}$. Let $0=s_0<s_1<\ldots<s_N=1$ be their endpoints (so that $\forall i:\, s_{i+1}-s_i = 1/N$). In what follows we always assume that indices are elements of $\Z/(N\Z)$ (i.e., for indices $i: N+i = i$).

Notice that if for some $k$ we have $s,t\in [s_{k-2}, s_{k+2}]$, then 
from \eqref{eqHolderSLThrConverse2} we get
\begin{equation}\label{eqThrHolderObsConstBound}
\frac{1}{2\sin\big(\pi[t-s]\big)}< \frac{\phi'(s)\phi'(t)}{\sin^2\Big(\pi \big[\phi(t)-\phi(s)\big]\Big)}<\frac{3}{2\sin\big(\pi[t-s]\big)}.
\end{equation}

Now, we consider a set of different gauge fixings. Let $\left\{\psi_k(t)\right\}_{k=0}^{N-1}$ be M\"{o}bius transformations such that, for all $k$,
\begin{equation}\label{eqThrHolderEquivConverseGaugeFixingSmallIntervals}
\psi_k\big(\phi(s_i)\big) = s_i, \qquad \text{for } i\in\{k-1, k, k+1\}.
\end{equation}

We proceed in several steps:

\medskip

\textbf{Step 1.} We show that $\forall k<N$ and $t\in [s_{k-2}, s_{k+2}]$; we have 
\begin{equation}
\frac{1}{C'_1}<\big(\psi_k \circ \phi\big)'(t)<C'_1.
\end{equation}

Indeed, from \eqref{eqThrHolderObsConstBound} and \eqref{eqThrHolderEquivConverseGaugeFixingSmallIntervals}, it follows that 
\begin{align}
\frac{1}{C'_2}<\Big(\psi_k \circ \phi\Big)'(s_{k-1})&\Big(\psi_k\circ \phi\Big)'(s_{k})<C'_2,\\ 
\frac{1}{C'_2}<\Big(\psi_k \circ \phi\Big)'(s_{k-1})&\Big(\psi_k\circ \phi\Big)'(s_{k+1})<C'_2,\\
\frac{1}{C'_2}<\Big(\psi_k \circ \phi\Big)'(s_{k})&\Big(\psi_k\circ \phi\Big)'(s_{k+1})<C'_2.
\end{align}
Therefore,
\begin{equation} \label{eqThrHolderEquivConverseCaseTwoDerivBound}
\frac{1}{C'_3}<\big(\psi_k \circ \phi\big)'(s_{j})<C'_3, \qquad \text{for } j\in\{k-1, k, k+1\}.
\end{equation}

Observe that for any $t \in [s_{k-1}, s_{k}]$ we can set $s = s_{k+1}$ in \eqref{eqThrHolderObsConstBound} and using \eqref{eqThrHolderEquivConverseCaseTwoDerivBound} get that
\begin{equation}
\frac{1}{C'_4} < \big(\psi_k\circ \phi\big)' (t) <C'_4.
\end{equation}
Similarly, we obtain the same estimate for $t\in  [s_{k}, s_{k+1}]$ if we take $s = s_{k-1}$, and for $t\in  [s_{k-2}, s_{k-1}]\cup [s_{k+1}, s_{k+2}]$ if we take $s = s_k$. This concludes Step 1.

\medskip

\textbf{Step 2.}
We show that all diffeomorphisms $\left\{\psi_k\circ \phi\right\}_{k=0}^{N-1}$ are not too far from $\phi$. That is, there exists a compact $K_1 = K_1(C')\subset \SL(2,\R)$, depending only on $C'$ such that $\forall k: \psi_k \in K_1$.
In particular, for all $t\in \T$, we have
\begin{equation}\label{eqTheHolderConverseCaseTwoDerivUnifBound}
\frac{1}{C'_{5}}<\phi'(t)<C'_{5}.
\end{equation}
Indeed, notice that \eqref{eqTheHolderConverseCaseTwoDerivUnifBound} follows from the existence of $K_1$ and Step 1. 
Now, we prove the existence of such $K_1$. 

From Step 1 it follows that, for any $k$,
\begin{equation}
\frac{1}{C'_{6}}|s_{k+2}-s_{k+1}|<\Big|\psi_k \circ \phi(s_{k+2})- \psi_k \circ \phi(s_{k+1})\Big|<C'_{6}|s_{k+2}-s_{k+1}|.
\end{equation}
Thus, there exists a compact $K_2 = K_2(C')\subset \SL(2,\R)$ such that, for all $k$, $\psi_{k+1}\circ \psi_{k}^{-1} \in K_2$. Combining this with Step 1 again, we obtain that, for all $k$ and all $t\in \T$,
\begin{equation}
\frac{1}{C'_{7}}<\big(\psi_k \circ \phi\big)' (t)<C'_{7}.
\end{equation}
Therefore, for all $k$ we have
\begin{align}
\frac{1}{C'_{8}}<\psi_k \circ \phi\left(\tfrac{1}{3}\right) &- \psi_k \circ \phi\left(0\right)<C'_{8}\\
\frac{1}{C'_{8}}<\psi_k \circ \phi\left(\tfrac{2}{3}\right) &- \psi_k \circ \phi\left(\tfrac{1}{3}\right)<C'_{8}\\
\frac{1}{C'_{8}}<\psi_k \circ \phi\left(1\right) &- \psi_k \circ \phi\left(\tfrac{2}{3}\right)<C'_{8},
\end{align}
which concludes this step.

\medskip

\textbf{Step 3.}
For each $s\in \T$, we let $\Psi_s$ be M\"{o}bius transformation of the unit circle such that 
\begin{equation}
\Psi_s\circ \phi (s) = 0, \qquad \big(\Psi_s\circ \phi\big)' (s) = 1, \qquad \Psi_s\circ \phi \left(s+\frac{1}{N}\right) = \frac{1}{N}.
\end{equation} 
We show that there exists a compact $K_3 = K_3(C')\subset \SL(2, \R)$ such that $\Psi_s \in K_3$. 
In particular, for all $s, t\in \T$, we have
\begin{equation}\label{eqThrHolderConverseCaseTwoDerivUnifBoundGauge}
\frac{1}{C'_{9}}<\big(\Psi_s \circ \phi\big)'(t)<C'_9.
\end{equation}
 
Indeed, fix $s\in \T$.
Let $k$ be such that $s\in[s_{k-1}, s_{k})$.
From Step 2 it is sufficient to show that there exists a compact $K_4(C') \subset \SL(2,\R)$ such that $\Psi_s \circ\psi_k^{-1}\in K_4$. This follows from Step 1. And \eqref{eqThrHolderConverseCaseTwoDerivUnifBoundGauge} now follows from \eqref{eqTheHolderConverseCaseTwoDerivUnifBound}.

\medskip

\textbf{Step 4.}
Now, we show that there exists $C'_{10}$ and $\delta' = \delta'(C')>0$ such that, for all $s\in \T$ and $\forall t \in [s,s+\delta']$,
\begin{equation}\label{eqTheHolderConverseCaseTwoDerivHolderGauge}
\Big|\log \big[ \big(\Psi_s\circ \phi\big)'(t)\big]-\log \big[ \big(\Psi_s\circ \phi\big)'(s)\big]\Big| \leq C'_{10} 
\,\dist{s}{t}^{\alpha}.
\end{equation} 

It follows from \eqref{eqHolderSLThrConverse2} that $\forall t\in [s, s+\frac{1}{N}]$; we have
\begin{equation}
\left|\frac{\big(\Psi_s\circ \phi\big)'(t)}{\sin^2\big(\pi\, \Psi_s\circ \phi(t)\big)} -\frac{1}{\sin^2\big(\pi[ t-s]\big)}\right| \leq C' 
\,\dist{s}{t}^{\alpha-2}.
\end{equation}
Denote $F_s(\tau) = \tan\Big(\pi\, \Psi_s\circ \phi(s+ \tau)\Big)$. 
Then it follows from the inequality above that, for all $\tau \in [0,1/N]$,
\begin{equation}\label{eqThrHolderConverseCaseTwoDiffBound}
\left|\frac{F'_s(\tau)}{F^2_s(\tau)}-\frac{\pi}{\sin^2(\pi \tau)}\right| \leq C'_{11} \tau^{\alpha-2}. 
\end{equation}
Integrating this from $x$ to $1/N$, we obtain
\begin{equation}
\left| \frac{1}{F_s(x)} - \cot(\pi x)\right| \leq C_{12}' x^{\alpha-1},
\end{equation}
for any $x\in (0, 1/N)$.

Notice that if $\delta'$ is sufficiently small, then for any $x\in (0, \delta')$ we get $x^{\alpha}C'_{12}<1/100$, and thus,
\begin{equation}\label{eqThrHolderConverseCaseTwoFBound}
F_s(x)<C'_{13} x,
\end{equation}
so
\begin{equation}
\left| 1 - F_s(x)\cot(\pi x)\right| \leq C_{14}' x^{\alpha}.
\end{equation}
Therefore, combining this with \eqref{eqThrHolderConverseCaseTwoFBound} again,
\begin{equation}
\left| 1 - \frac{F_s(x)}{\sin(\pi x)} \right| \leq C_{15}' x^{\alpha}
\end{equation}
for all $x\in (0, \delta')$.
Combining this with \eqref{eqThrHolderConverseCaseTwoDiffBound} and \eqref{eqThrHolderConverseCaseTwoFBound}, we obtain
\begin{equation}
\left| F'_s(\tau) - \pi\right| \leq C'_{16}\tau^{\alpha}
\end{equation}
for all $\tau \in [0, \delta']$.
Therefore, using \eqref{eqThrHolderConverseCaseTwoDerivUnifBoundGauge} for $\tau \in [0 ,\delta']$, we get
\begin{equation}
\left|\log F'_s(\tau) - \log F'_s(0)\right|\leq C'_{17}\tau^{\alpha},
\end{equation}
which together with \eqref{eqThrHolderConverseCaseTwoDerivUnifBoundGauge} implies \eqref{eqTheHolderConverseCaseTwoDerivHolderGauge}.

\medskip

\textbf{Step 5.} Combination of \eqref{eqTheHolderConverseCaseTwoDerivHolderGauge} with Step 3 finishes the proof of the theorem.
\end{proof}

\section{Weak Large Deviations Principle}\label{sectWeakLDP}

First, we show that $\I(\cdot)$ is a rate function.

\begin{prp}\label{prpActionIsRateFunction}
Functional $\I(\cdot)$ is a rate function. That is, $\I(\cdot)$ is lower semicontinuous on $\Diff^1(\T)/\SL(2, \R)$.
In other words, for any $\alpha \in \R$, the level set 
\begin{equation}
\Psi(\alpha) = \left\{\phi\in \Diff(\T)/\SL(2, \R): \I(\phi)\leq \alpha \right\}
\end{equation} is closed.
\end{prp}
\begin{proof}
Recall that $\I(\phi) = \frac{1}{2}\int_{\T} \big(\log \phi'(\tau)\big)'^{\, 2}\d\tau -2\pi^2\int_{\T}\phi'^{\, 2}(\tau)\d\tau$.
It is well known that a sum of two lower semicontinuous functions is lower semicontinuous. 

First, we notice that the first term 
$\frac{1}{2}\int_{\T} (\log \phi'(\tau))'^{\, 2}\d\tau$ 
is lower semi-continuous on $\Diff^1(\T)$ with respect to $C^1$ topology.  
This follows from the fact that $\int_{\T} |f'(\tau)|^2\d\tau$ is lower semi-continuous on $C[0,1]$ with supremum norm topology (see, e.g. \citep[Theorem 5.2.3]{bookLDP}) applied to $f=\log \phi'$. 
Here we also use that supremum norms for both $\phi'$ and $\log \phi'$ induce the same topology on $\Diff^1(\T)$.
The second term $2\pi^2\int_{\T}\phi'^{\, 2}(\tau)\d\tau$ is a continuous function on $\Diff^1(\T)$, since we are using $C^1$ topology.
\end{proof}

Second, we prove the weak large deviations principle.
It differs from the large deviations principle only in that the upper bound holds only for compact sets, instead of closed sets.
\begin{prp}\label{prpWeakLDP}
The family of measures $\left\{\d\FMeas{\sigma^2}\right\}_{\sigma>0}$ satisfies weak large deviations principle with rate function $\I(\cdot)$ as $\sigma\to 0$. In other words, the following holds:
\begin{enumerate}
	\item For any open set $U\subset \Diff^1(\T)/\SL(2,\R)$, we have
\begin{equation}
\liminf_{\sigma\to 0} \sigma^2 \log \FMeas{\sigma^2}(U) \geq -\inf_{x\in U} \I(x).
\end{equation}
	\item For any compact set $V\subset \Diff^1(\T)/\SL(2,\R)$, we have
\begin{equation}
\limsup_{\sigma\to 0} \sigma^2 \log \FMeas{\sigma^2}(V) \leq -\inf_{x\in V} \I(x).
\end{equation}
\end{enumerate} 
\end{prp}

\begin{proof}
The action $\I(\cdot)$ is a rate function by Proposition \ref{prpWeakLDP}, so it is sufficient to give bounds on the measure.

Notice that $\Diff^1(\T)$ is topologically isomorphic to $\DiffG^1(\T)\times \SL(2,\R)$. 
Throughout this proof we slightly abuse the notation and identify this two spaces.

First, we prove the lower bound. Let $U \subset \DiffG^1(\T)$ be an open set. Fix an open set $A\subset\SL(2,\R)$ with Haar measure equal to $1$. Then from \eqref{defMeasureMeas}, \eqref{defMeasureMu} and Proposition \ref{propMeasureFactor} applied to indicator function of $U\times A$, we get
\begin{equation}
\FMeas{\sigma^2}(U) =\frac{1}{\sqrt{2\pi}\sigma} 
\int_{\A^{-1}(U\times A)}
\exp\left\{ \frac{2\pi^2 }{\sigma^2}\int_0^1 \big[\A_{\xi}'(\tau)\big]^2\d\tau\right\}
\d\WS{\sigma^2}{0}{1}(\xi)
\end{equation}
Define
\begin{equation}
F(\xi) = 
\begin{cases}	2\pi^2 \int_0^1 \big[\A_{\xi}'(\tau)\big]^2\d\tau, &\qquad \text{if } \xi \in \A^{-1}(U\times A)\\
-\infty,  &\qquad \text{otherwise.}
\end{cases} 
\end{equation}
Observe that since both $U$ and $A$ are open, $F$ is lower semicontinuous. Moreover,
\begin{equation}
\FMeas{\sigma^2}(U) =\frac{1}{\sqrt{2\pi}\sigma} 
\int_{C[0, 1]}
\exp\left\{ \frac{1}{\sigma^2}F(\xi)\right\}
\d\WS{\sigma^2}{0}{1}(\xi)
\end{equation}
Notice that $\d\WS{\sigma^2}{0}{1}(\xi)$ satisfies Large Deviations Principle with good rate function $\frac{1}{2}\int_0^1 |\xi'(\tau)|^2\d\tau$ (see, e.g. \citep[Theorem 3.4.12]{Deuschel_Stroock_LDP}). 
Then, using \citep[Lemma 4.3.4]{bookLDP},
\begin{equation}
\liminf_{\sigma\to 0}\sigma^2 \log \FMeas{\sigma^2}(U)\geq -\inf_{\xi \in \A^{-1}(U\times A)}\left(\frac{1}{2}\int_0^1 |\xi'(\tau)|^2\d\tau- F(\xi)\right).
\end{equation}
The claim now follows from 
Proposition \ref{prpActionSLInvariance}
and the fact that
\begin{equation}\label{eqLemmaWeakLDPRateFunctionFormula}
\I(\phi) = \frac{1}{2}\int_0^1 \Big|\big(\A^{-1}\phi\big)'(\tau)\Big|^2\d\tau- F\left(\A^{-1}\phi\right),
\end{equation}
for $\phi\in U\times A$.

Finally, we prove the upper bound. The proof is analogous to the lower bound.
Let $V \subset \DiffG^1(\T)$ be a compact set. Fix a compact set $A\subset\SL(2,\R)$ with Haar measure equal to $1$. 
Then from \eqref{defMeasureMeas}, \eqref{defMeasureMu}, and Proposition \ref{propMeasureFactor} applied to indicator function of $V\times A$, we get
\begin{equation}
\FMeas{\sigma^2}(V) =\frac{1}{\sqrt{2\pi}\sigma} 
\int_{\A^{-1}(V\times A)}
\exp\left\{ \frac{2\pi^2 }{\sigma^2}\int_0^1\big[\A_{\xi}'(\tau)\big]^{2}\d\tau\right\}
\d\WS{\sigma^2}{0}{1}(\xi)
\end{equation}

Define
\begin{equation}
F(\xi) = 
\begin{cases}	2\pi^2 \int_0^1 \big[\A_{\xi}'(\tau)\big]^{2}\d\tau, &\qquad \text{if } \xi \in \A^{-1}(V\times A)\\
-\infty,  &\qquad \text{otherwise.}
\end{cases} 
\end{equation}
Observe that, since both $V$ and $A$ are closed, $F$ is upper semicontinuous. 
Moreover, since both $\A$ and $\A^{-1}$ are continuous and both $V$ and $A$ are compact, we get that $F(\xi)$ is bounded from above. 
Furthermore,
\begin{equation}
\FMeas{\sigma^2}(V) =\frac{1}{\sqrt{2\pi}\sigma} 
\int_{C[0, 1]}
\exp\left\{ \frac{1}{\sigma^2}F(\xi)\right\}
\d\WS{\sigma^2}{0}{1}(\xi)
\end{equation}
We know that $\d\WS{\sigma^2}{0}{1}(\xi)$ satisfies Large Deviations Principle with good rate function $\frac{1}{2}\int_0^1 |\xi'(\tau)|^2\d\tau$. Therefore, using \citep[Lemma 4.3.6]{bookLDP},
\begin{equation}
\limsup_{\sigma\to 0}\sigma^2 \log \FMeas{\sigma^2}(V)\leq -\inf_{\xi \in \A^{-1}(V\times A)}\left(\frac{1}{2}\int_0^1 |\xi'(\tau)|^2\d\tau- F(\xi)\right).
\end{equation}
The statement now follows from
Proposition \ref{prpActionSLInvariance} 
and \eqref{eqLemmaWeakLDPRateFunctionFormula}.
 
\end{proof}

\section{Exponential estimates}
\label{sectExpEstimates}
The main goal of this subsection is to prove Theorem \ref{thrHolderClassExpTail}. 

\subsection{Exponential moments}
First, we estimate exponential moments of the observables which appear in the definition of our H\"{o}lder condition.
This estimate is sharper than the one in \citep{LosevCorr} when $s$ is close to $t$.
We also use this estimate to prove an exponential tale estimate.
\begin{prp}\label{prpBoundExpMoments}
There exists $C>0$ such that for any $\sigma>0$, any $s, t\in \T$ such that $t-s\leq 1/2$,
and any $z\in [-C^{-1}, C^{-1}]$, we have
\begin{equation}\label{eqMomentGenFuncBound}
\int \exp\left\{\frac{z\sqrt{(t-s)}}{\sigma^2} \left(\obs{\phi}{s}{t} -\frac{1}{(t- s)}\right)\right\}\d \FMeas{\sigma^2}(\phi) 
< C\, \exp\left(\frac{C}{\sigma^2}\right).
\end{equation}
\end{prp}
\begin{proof}

Using Proposition \ref{prpMainPartFunct} and Proposition \ref{prpMainObsMoments} we can express exponential moments of $\obs{\phi}{s}{t}$, as
\begin{multline}
\int \exp\left\{\frac{2 \lambda}{\sigma^2}\, \obs{\phi}{s}{t} \right\}\d \FMeas{\sigma^2}(\phi) 
 = \int_{\R_+} \exp\left(-\frac{\sigma^2}{2}\cdot k_2^2\right)\sinh(2\pi k_2)\, 2k_2\d k_2 \\
 +
\sum_{l=1}^{\infty}\int_{\R_+^{2}} 
\frac{\Gamma\big(\frac{l}{2} \pm i k_1 \pm i k_2\big)}{2\pi^2\, \Gamma(l)}\cdot\frac{\lambda^l}{l!}\cdot
\exp\left(-\frac{(t-s)\sigma^2}{2}\cdot k_1^2 -\frac{\big(1-(t-s)\big)\sigma^2}{2}\cdot k_2^2\right)  
\\
\times
\sinh(2\pi k_1)\, 2 k_1 \sinh(2\pi k_2)\, 2 k_2 \d k_1 \d k_2,
\end{multline}
for $\lambda\in [-1, 1]$. Here the right-hand side converges absolutely, because it is dominated by the same expression with $\lambda=1$, which converges because all terms are positive and exponential moments of $\obs{\phi}{s}{t}$ exist (see Proposition~\ref{prpExpMoment}).

Because of rotational invariance in the argument we can assume that $s=0$.

First, we give an upper bound for the integral over $k_1$.
It follows from Proposition~\ref{prpKeyArccoshExpansion} that
\begin{multline}\label{eq_prp_bound_exp_moment_gener_func}
\int_{\R_+}\sum_{l=1}^{\infty}\frac{\Gamma\big(\frac{l}{2} \pm i k_1 \pm i k_2\big)}{2\pi^2\, \Gamma(l)}\cdot\frac{\lambda^l }{l!}\cdot \exp\left\{-\frac{t \sigma^2}{2}\cdot k^2_1\right\}\sinh(2\pi k_1)\, 2 k_1 \d k_1\\
=
\frac{2}{\pi}\int_{\R_+}\int_{\R_+} \exp\left\{-\frac{t\sigma^2}{2}\cdot k^2_1\right\}
\Bigg(\cos\Big(2k_1\, \arccosh\left[\cosh\left(\tfrac{\beta}{2}\right)-\tfrac{\lambda}{2}\right]\Big)- \cos(k_1\beta)\Bigg)
\cos\left(k_2 \beta\right) \d \beta \d k_1 .
\end{multline}
Notice that by Statement~\ref{stmArccoshTaylor}
\begin{equation}
\left|\cos\Big(2k_1\, \arccosh\left[\cosh\left(\tfrac{\beta}{2}\right)-\tfrac{\lambda}{2}\right]\Big)- \cos(k_1\beta)\right|
=
O\left(e^{10k_1-\beta/10}\right),
\end{equation}
and so the integral in the right-hand side of \eqref{eq_prp_bound_exp_moment_gener_func} converges absolutely.
We calculate the integral in $k_1$ first,
\begin{equation}
\frac{2}{\pi}
\int_{\R_+}\int_{\R_+} \exp\left\{-\frac{t\sigma^2}{2}\cdot k^2_1\right\}
\cos(k_1\beta)
\d k_1
\cos\left(k_2 \beta\right) \d \beta  =  \exp\left\{-\frac{t\sigma^2}{2}\cdot k^2_2\right\},
\end{equation}
and
\begin{multline}
\frac{2}{\pi}\int_{\R_+}\int_{\R_+} \exp\left\{-\frac{t\sigma^2}{2}\cdot k^2_1\right\}\cos\left(2k_1\, \arccosh\left[\cosh\left(\frac{\beta}{2}\right)-\frac{\lambda}{2}\right]\right)
\d k_1 \cos\left(k_2 \beta\right) \d \beta\\
=
\frac{\sqrt{2}}{\sqrt{\pi t}\sigma}
\int_{\R_+}
\exp\left\{ -\frac{2}{t\sigma^2}
\arccosh^2\left[\cosh\left(\frac{\beta}{2}\right)-\frac{\lambda}{2}\right]\right\}\cos\left(k_2 \beta\right) \d \beta\\
\leq
\frac{\sqrt{2}}{\sqrt{\pi t}\sigma}
\int_{\R_+}
\exp\left\{ -\frac{2}{t\sigma^2}
\arccosh^2\left[\cosh\left(\frac{\beta}{2}\right)-\frac{\lambda}{2}\right]\right\}\d \beta.
\end{multline}
Using Statement~\ref{stmArccoshTaylor}, we get that for some $C_1>0$
\begin{equation}
\int_{\R_+}
\exp\left\{ -\frac{2}{t\sigma^2}
\arccosh^2\left[\cosh\left(\frac{\beta}{2}\right)-\frac{\lambda}{2}\right]\right\}\d \beta
\leq
\int_{\R_+}
\exp\left\{ -\frac{\beta^2}{2t\sigma^2}
+\frac{2\lambda}{t\sigma^2} + \frac{C_1|\lambda|(\beta^2+ |\lambda|)}{2t\sigma^2}
\right\}\d \beta.
\end{equation}
Therefore,
\begin{multline}
\exp\left\{-\frac{t\sigma^2}{2}\cdot k^2_2\right\}+
\int_{\R_+}\sum_{l=1}^{\infty}\frac{\Gamma\big(\frac{l}{2} \pm i k_1 \pm i k_2\big)}{2\pi^2\, \Gamma(l)}\cdot\frac{\lambda^l }{l!}\cdot \exp\left\{-\frac{t \sigma^2}{2}\cdot k^2_1\right\}\sinh(2\pi k_1)\, 2 k_1 \d k_1 \\
 \leq
 \frac{\sqrt{2}}{\sqrt{\pi t}\sigma}
 \int_{\R_+}
\exp\left\{ -\frac{\beta^2}{2t\sigma^2}
+\frac{2\lambda}{t\sigma^2} + \frac{C_1|\lambda|(\beta^2+ |\lambda|)}{2t\sigma^2}
\right\}\d \beta.
\end{multline}

Hence, taking $\lambda = z\sqrt{t}/2$, we obtain that, for some $C_2>0$ and for all $z\in[-C_1^{-1}, C_1^{-1}]$, 
\begin{multline}
\int \exp\left\{\frac{z\sqrt{t}}{\sigma^2} \left(\obs{\phi}{0}{t} -\frac{1}{t}\right)\right\}\d \FMeas{\sigma^2}(\phi) \\
\leq
\frac{\sqrt{2}}{\sqrt{\pi t}\sigma}
\int_{\R_+^2}
\exp\left\{ -\frac{\beta^2}{2t\sigma^2}
+ \frac{C_1 z^2}{4\sigma^2}
+ \frac{C_1 z \beta^2}{4\sqrt{t}\sigma^2}
\right\}
\exp\left(-\frac{(1-t)\sigma^2}{2}\cdot k_2^2\right)
\sinh(2\pi k_2)\, 2k_2 \d k_2
\d \beta
\\
\leq 
C_2 \exp\left(\frac{C_2z^2}{\sigma^2}\right) \PartFunc{\sigma^2/2},
\end{multline}
which, together with the explicit formula for $\PartFunc{\sigma^2/2}$ in Proposition~\ref{prpMainPartFunct}, finishes the proof.
\end{proof}

\begin{crl}\label{crlObsTailBound}
There exists $C>0$ such that for any $M, \eps>0$, $\sigma>0$, and $s\neq t$,
\begin{multline}
\FMeas{\sigma^2} \left(\phi \in \Diff^1(\T)/\SL(2, \R)\,:\, \left|\obs{\phi}{s}{t}-\frac{\pi}{\sin\Big(\pi[t-s]\Big)}\right|>M\, \dist{s}{t}^{-\tfrac{1}{2}-\eps}\right)\\
<C\exp\left(\frac{C-C^{-1}M\, \dist{s}{t}^{-\eps}}{\sigma^2}\right).
\end{multline}
\end{crl}
\begin{proof}
Follows from Proposition \ref{prpBoundExpMoments}, the fact that for some $C'>0$ we have
\begin{equation}
\left|\frac{\pi}{\sin\Big(\pi [t-s]\Big)} - \frac{1}{t-s}\right| \leq C',
\end{equation} 
and Markov's Inequality for moment generating functions.
\end{proof}

\subsection{Interpolation}
We start by proving two lemmas which allow us to relate observables calculated at different points to each other.

\begin{lmm}\label{lmmObsAlgEquality}
Let $\tau_1, \tau_2, \tau_3, \tau_4$ be distinct points that lie on the circle either in clockwise or anticlockwise order. 
Then
\begin{equation}\label{lmmObsAlgEqualityEq}
\frac{1}{\obs{\phi}{\tau_1}{\tau_3}\obs{\phi}{\tau_2}{\tau_4}} = 
\frac{1}{\obs{\phi}{\tau_1}{\tau_2}\obs{\phi}{\tau_3}{\tau_4}}+\frac{1}{\obs{\phi}{\tau_1}{\tau_4}\obs{\phi}{\tau_2}{\tau_3}}
\end{equation}
\end{lmm}
\begin{proof}
Notice that if $u, v, w$ lie on $\T$ in anticlockwise order then
\begin{multline}
\frac{\obs{\phi}{u}{v}\obs{\phi}{u}{w}}{\obs{\phi}{v}{w}} 
= \frac{\pi\phi'(u)\sin\Big(\pi\big[\phi(w)-\phi(v)\big]\Big)}{\sin\Big(\pi\big[\phi(v)-\phi(u)\big]\Big)\sin\Big(\pi\big[\phi(w)-\phi(u)\big]\Big)} \\
= \pi\phi'(u)\Biggl( \cot \Big(\pi\big[\phi(v)-\phi(u)\big]\Big) -  \cot \Big(\pi\big[\phi(w)-\phi(u)\big]\Big) 
\Biggl).
\end{multline}
Therefore,
\begin{equation}
\frac{\obs{\phi}{\tau_1}{\tau_2}\obs{\phi}{\tau_1}{\tau_3}}{\obs{\phi}{\tau_2}{\tau_3}} 
+\frac{\obs{\phi}{\tau_1}{\tau_4}\obs{\phi}{\tau_1}{\tau_3}}{\obs{\phi}{\tau_3}{\tau_4}} 
=
\frac{\obs{\phi}{\tau_1}{\tau_2}\obs{\phi}{\tau_1}{\tau_4}}{\obs{\phi}{\tau_2}{\tau_4}},
\end{equation}
which is equivalent to \eqref{lmmObsAlgEqualityEq}.
\end{proof}

\begin{lmm}\label{lmmObsIneq}
Assume that $s, s_1, t_1, t, t_2$ are points that lie on the circle either in clockwise or anticlockwise order.
Then
\begin{equation}
\frac{\obs{\phi}{s}{t_1}\obs{\phi}{s_1}{t}}{\obs{\phi}{s_1}{t_1}}
<\obs{\phi}{s}{t}
< \frac{\obs{\phi}{s}{t_2}\obs{\phi}{s_1}{t}}{\obs{\phi}{s_1}{t_2}}.
\end{equation}
\end{lmm}
\begin{proof}

It follows from Lemma \ref{lmmObsAlgEquality} that
\begin{equation}
\frac{1}{
\obs{\phi}{s}{t}} 
= \obs{\phi}{s_1}{t_2}\Biggl(\frac{1}{\obs{\phi}{s}{t_2}\obs{\phi}{s_1}{t}}
+\frac{1}{\obs{\phi}{s}{s_1}\obs{\phi}{t}{t_2}}\Biggl)
>\frac{\obs{\phi}{s_1}{t_2}}{\obs{\phi}{s}{t_2}\obs{\phi}{s_1}{t}},
\end{equation}
and
\begin{equation}
\frac{1}{\obs{\phi}{s}{t}}
=
\obs{\phi}{s_1}{t_1}
\Biggl(\frac{1}{\obs{\phi}{s}{t_1}\obs{\phi}{s_1}{t}}
-\frac{1}{\obs{\phi}{s}{s_1}\obs{\phi}{t_1}{t}}
\Biggl) 
< \frac{\obs{\phi}{s_1}{t_1}}{\obs{\phi}{s}{t_1}\obs{\phi}{s_1}{t}},
\end{equation}
which finishes the proof.
\end{proof}

\subsection{Exponential concentration}
Recall the statement of Theorem~\ref{thrHolderClassExpTail}.
\begin{customthm}{\ref{thrHolderClassExpTail}}
For any $\alpha\in [1/4, 1/2)$ and any $\Lambda, N>0$, there exists $M>0$ such that for any $\sigma\in (0, \Lambda)$,
\begin{equation}\label{eqExpBoundAllPoints2}
\FMeas{\sigma^2} \Big(
\MyHol{\alpha}{M}
\Big)
\\
\geq \PartF(\sigma^2)- \exp\left(-\frac{N}{\sigma^2}\right).
\end{equation}
\end{customthm}

\begin{proof}[Proof of Theorem \ref{thrHolderClassExpTail}]
Let $\Dyadic{n} = \Big\{k/2^n \, : \, k\in \{1, 2, 3, \ldots 2^n-1\}\Big\}$. 

Let $K>1$ be a large number to be chosen later, and denote $R = 10^{10} K$.
Let $\mathsf{A}_n = \mathsf{A}_n(K)$ be an event
\begin{multline}
\mathsf{A}_n = \Biggl\{\phi\in\Diff^1(\T)/\SL(2, \R) \, :\, \left|\log \left[\obs{\phi}{s}{t}\right]-\log \left[\frac{\pi}{\sin\big(\pi[t-s]\big)}\right]\right|<K \, \dist{s}{t}^{\alpha}, \\
\text{for all } s,t\in \Dyadic{n} \text{ such that } \dist {s}{t} \leq 20 \cdot 2^{-n/2} 
\Biggl\}.
\end{multline}

\medskip

\textbf{Step 1.}
We start by giving a lower bound for $\FMeas{\sigma^2}\big(\cap_{n=R}^{\infty}\mathsf{A}_n\big)$.
Denote
\begin{multline}
\mathsf{B}_n = \Biggl\{\phi\in\Diff^1(\T)/\SL(2, \R) \, :\, \left|\obs{\phi}{s}{t}-\frac{\pi}{\sin\big(\pi[t-s]\big)}\right|< \frac{1}{10} \, K \, \dist{s}{t}^{\alpha-1},\\
\text{ for all } s,t\in \Dyadic{n} \text{ such that } \dist{s}{t} \leq 20 \cdot 2^{-n/2}\Biggl\}.
\end{multline}

It follows from the inequalities $\forall x\in (-1/2, 1/2):\, \left|\log (1+x)\right|\leq 10 |x|$ and $\left|\sin (x)\right|\leq |x|$ that for $n\geq R$, we have $\mathsf{B}_n \subset \mathsf{A}_n$.
Moreover, it follows from Corollary \ref{crlObsTailBound} that for some $C>0$ (possibly different from $C$ in  Corollary \ref{crlObsTailBound}),
\begin{equation}
\FMeas{\sigma^2}\big(\mathsf{B}_n\big)
\geq 1 - C \, 2^{3n/2}\exp\left(\frac{C-C^{-1}\, K\, 2^{\left(\tfrac{1}{2}-\alpha\right)n/2}}{\sigma^2}\right).
\end{equation}
Note that the factor $2^{3n/2}$ comes from the fact that $\mathsf{B}_n$ is defined by $20\cdot 2^{3n/2}$ inequalities.

Since $\alpha<1/2$, we deduce that for any $N>0$ we can find $K>1$ such that, for any $\sigma \in (0, \Lambda)$, 
\begin{equation}\label{eqCapALowerBound}
\FMeas{\sigma^2}\Big(\cap_{n=R}^{\infty} \mathsf{A}_n\Big)
\geq
\FMeas{\sigma^2}\Big(\cap_{n=R}^{\infty} \mathsf{B}_n\Big)
\geq \PartF(\sigma^2) -\exp\left(-\frac{N}{\sigma^2}\right).
\end{equation}

\medskip
\textbf{Step 2.} Second, we show that if $\phi \in \cap_{n=R}^{\infty} \mathsf{A}_n$, then for any $m\geq n\geq R$, any $s\in \Dyadic{n}$ and any $t \in \Dyadic{m}$ with $\dist{s}{t}\leq 10 \cdot 2^{-n}$, we have
\begin{equation}\label{eqLogObsApproxInduction}
 \left|\log \left[\obs{\phi}{s}{t}\right]-\log \left[\frac{\pi}{\sin\big(\pi[t-s]\big)}\right]\right|
 < 10^6  K  \, \dist{s}{t}^{\alpha}.
\end{equation}

We prove this by reverse induction on $n$. We also assume that $t\notin\Dyadic{m-1}$.

\textit{Base Case:} $n=m$.
This follows from the fact that $\phi\in \mathsf{A}_m$.

\textit{Induction Step:} $n+1 \mapsto n$.

If $\dist{s}{t}\leq 10\cdot 2^{-n-1}$, then we just apply the induction hypothesis. 
Therefore, now we can assume $\dist{s}{t}> 10\cdot 2^{-n-1}$.

Let $s_1, s_2 \in \Dyadic{n+1}$ be two consecutive points in $\Dyadic{n+1}$ such that $t$ lies on the dyadic interval of size $2^{-n-1}$ between them. 
We assume that $\dist{s_1}{t}\geq \dist{s_2}{t}$. 
In particular, $2^{-n-2} \leq \dist{s_1}{t}\leq 2^{-n-1}$.

If $2n< m$, we let $t_1, t_2\in \Dyadic{2n}$ be two points such that $t$ lies on the dyadic interval of size $2^{-2n}$ between them. 
Moreover, we assume that $t_1$ lies between $s_1$ and $t$. 
In other words, $s, s_1, t_1, t, t_2$ lie on the circle either in clockwise or anticlockwise order.
If $2n\geq  m$, we just put $t_1 = t_2 = t$.
Now we can apply Lemma \ref{lmmObsIneq}.

\medskip

First, we give an upper bound for $\log \left[\obs{\phi}{s}{t}\right]$.
From Lemma \ref{lmmObsIneq} we deduce that
\begin{equation}\label{eqLogObsBoundAboveInduction}
\log \left[\obs{\phi}{s}{t}\right]
\leq 
\log \left[\obs{\phi}{s_1}{t}\right]
+\log \left[\obs{\phi}{s}{t_2}\right]
- \log \left[\obs{\phi}{s_1}{t_2}\right].
\end{equation}

Using induction hypothesis and the fact that $\dist{s_1}{t}\leq 2^{-n-1}$, we get
\begin{multline}\label{eqLogObsBoundAboveInductionFirstTerm}
\log \left[\obs{\phi}{s_1}{t}\right]
\leq \log\left[\frac{\pi}{\sin\big(\pi[t-s_1]\big)}\right] 
+ 10^6 K\, \dist{s_1}{t}^{\alpha}\\
\leq
 \log\left[\frac{\pi}{\sin\big(\pi[t-s_1]\big)}\right] 
+10^6 K\, 2^{-\alpha(n+1)}.
\end{multline}
Moreover, since $\phi \in \mathsf{A}_{2n}$, $s, t_2\in\Dyadic{2n}$ and 
$\dist{s}{t_2}\leq \dist{s}{t}+\dist{t_1}{t_2}\leq 10\cdot 2^{-n} +2^{-2n} \leq 11\cdot 2^{-n}$, we have
\begin{equation}
\log \left[\obs{\phi}{s}{t_2}\right]
\leq
\log\left[\frac{\pi}{\sin\big(\pi[t_2-s]\big)}\right] 
+  K\, \dist{s}{t_2}^{\alpha}.
\end{equation}
Combining this with the fact that $10 \cdot 2^{-n-1}<\dist{s}{t}<10 \cdot 2^{-n}$, inequality $\dist{t}{t_2}\leq\dist{t_1}{t_2}\leq 2^{-2n}$, and Statement \ref{stmLogSinTriangle}, we obtain
\begin{equation}\label{eqLogObsBoundAboveInductionSecondTerm}
\log \left[\obs{\phi}{s}{t_2}\right]
\leq
\log\left[\frac{\pi}{\sin\big(\pi[t-s]\big)}\right] 
+1000\cdot 2^{-n}
+11 K\, 2^{-\alpha n}
.
\end{equation}
Similarly, since $\phi \in \mathsf{A}_{2n}$, $s_1, t_2\in\Dyadic{2n}$ and 
$\dist{s_1}{t_2}\leq \dist{s_1}{s_2}\leq 2^{-n-1}$, we get
\begin{equation}
\log \left[\obs{\phi}{s_1}{t_2}\right] 
\geq 
\log\left[\frac{\pi}{\sin\big(\pi[t_2-s_1]\big)}\right] 
- K\, \dist{s_1}{t_2}^{\alpha}.
\end{equation}
Combining this with the inequalities $2^{-n-2} \leq \dist{s_1}{t}\leq 2^{-n-1}$ and $\dist{t}{t_2}\leq\dist{t_1}{t_2}\leq 2^{-2n}$ and Statement~\ref{stmLogSinTriangle}, we obtain
\begin{equation}\label{eqLogObsBoundAboveInductionThirdTerm}
\log \left[\obs{\phi}{s_1}{t_2}\right] 
\geq
\log\left[\frac{\pi}{\sin\big(\pi[t-s_1]\big)}\right] 
- 4000\cdot 2^{-n}
- K\, 2^{-\alpha n}.
\end{equation}

Using \eqref{eqLogObsBoundAboveInduction}, \eqref{eqLogObsBoundAboveInductionFirstTerm}, \eqref{eqLogObsBoundAboveInductionSecondTerm}, and \eqref{eqLogObsBoundAboveInductionThirdTerm}, we get
\begin{equation}\label{eqLogObsBoundAboveInductionLastStep}
\log \left[\obs{\phi}{s}{t}\right]
\leq 
\log\left[\frac{\pi}{\sin\big(\pi[t-s]\big)}\right]  
+ \Big(10^6\cdot 2^{-\alpha}+10^5\Big)K\, 2^{-\alpha n}.
\end{equation}
Since $\alpha\geq 1/4$, we have $2^{-\alpha}\leq 2^{-1/4} < \frac{9}{10}$, so 
\begin{equation}\label{eqLogObsBoundAboveInductionErrorUpperBound}
\Big(10^6\cdot 2^{-\alpha}+10^5\Big)K\, 2^{-\alpha n}
< 10^6 K\, 2^{-\alpha n}<10^6 K\, \dist{s}{t}^{\alpha}.
\end{equation}
Combining \eqref{eqLogObsBoundAboveInductionLastStep} and \eqref{eqLogObsBoundAboveInductionErrorUpperBound}, we deduce that
\begin{equation}
\log \left[\obs{\phi}{s}{t}\right]
\leq 
\log\left[\frac{\pi}{\sin\big(\pi[t-s]\big)}\right] 
+10^6 K\, \dist{s}{t}^{\alpha}.
\end{equation}

\medskip

Second, analogously to the upper bound for $\log \left[\obs{\phi}{s}{t}\right]$, we can bound it from below using Lemma \ref{lmmObsIneq} by
\begin{equation}
\log \left[\obs{\phi}{s}{t}\right] 
\geq 
\log\left[\frac{\pi}{\sin\big(\pi[t-s]\big)}\right] 
-10^6 K\, \dist{s}{t}^{\alpha}.
\end{equation}
This finishes the proof of the induction step. Thus, we have proved \eqref{eqLogObsApproxInduction}.

\medskip

\textbf{Step 3.} Now, we show that if $\phi \in \cap_{n=R}^{\infty} \mathsf{A}_n$, then for any $s, t \in \cup_{n=R}^{\infty} \Dyadic{n}$ with $\dist{s}{t}\leq 2^{-R}$, we have
\begin{equation}\label{eqLogObsApproxUnifStep2Goal}
 \left|\obs{\phi}{s}{t}-\frac{\pi}{\sin\big(\pi[t-s]\big)}\right|
 < e^{10^8 K}\, \dist{s}{t}^{\alpha-1}.
\end{equation}

Let $n\geq R$ be such that $2^{-n+2} < \dist{s}{t}\leq 2^{-n+3}$.
Also, let $m\geq n$ be such that $s, t\in \Dyadic{m}$. 
Since $ \dist{s}{t}> 4\cdot 2^{-n}$, there exist $r_1, r_2\in \Dyadic{n}$ between $s$ and $t$ (on the shorter arc connecting $s$ and $t$) such that $\dist{r_1}{r_2} = 2^{-n}$, $ \dist{s}{r_1}\geq 2^{-n}$, $\dist{r_2}{t}\geq 2^{-n}$,
and such that points $s, r_1, r_2, t$ lie on $\T$, either in clockwise or anticlockwise order.
Applying Lemma \ref{lmmObsAlgEquality}, we get
\begin{equation}\label{eqObsApproxEquality}
\frac{1}{\obs{\phi}{s}{t}} = \obs{\phi}{r_1}{r_2}\Biggl(\frac{1}{\obs{\phi}{s}{r_2}\obs{\phi}{r_1}{t}}
-\frac{1}{\obs{\phi}{s}{r_1}\obs{\phi}{r_2}{t}}\Biggl)
\end{equation}
Since $\dist{s}{t}\leq 8\cdot 2^{-n}$, we can apply \eqref{eqLogObsApproxInduction} to estimate all terms in the right-hand side of equation above. 
Notice that, for any $\gamma\in [0,1]$ and any $x\in \R$, we have $|e^{\gamma x}-1|\leq |\gamma| e^{|x|}$. 
Combining this inequality for $\gamma = \dist{s}{t}^{\alpha}$ with \eqref{eqLogObsApproxInduction}, we obtain
\begin{multline}\label{eqObsApproxEqualityFirstTermBound}
\left|\frac{\obs{\phi}{r_1}{r_2}}{\obs{\phi}{s}{r_2}\obs{\phi}{r_1}{t}}
-\frac{\sin\big(\pi[r_2-s]\big)\sin\big(\pi[t-r_1]\big)}{\pi \sin\big(\pi[r_2-r_1]\big)}\right|\\
\leq \frac{\sin\big(\pi[r_2-s]\big)\sin\big(\pi[t-r_1]\big)}{\pi \sin\big(\pi[r_2-r_1]\big)} \, e^{10^7 K+10}\, 2^{-\alpha n}
\leq e^{10^7 K+50}\, 2^{-(\alpha+1)n}.
\end{multline}
Similarly,
\begin{multline}\label{eqObsApproxEqualitySecondTermBound}
\left|\frac{\obs{\phi}{r_1}{r_2}}{\obs{\phi}{s}{r_1}\obs{\phi}{r_2}{t}}
-\frac{\sin\big(\pi[r_1-s]\big)\sin\big(\pi[t-r_2]\big)}{\pi \sin\big(\pi[r_2-r_1]\big)}\right|\\
\leq \frac{\sin\big(\pi[r_1-s]\big)\sin\big(\pi[t-r_2]\big)}{\pi \sin\big(\pi[r_2-r_1]\big)}\, e^{10^7 K+10}\, 2^{-\alpha n}
\leq e^{10^7 K+50}\, 2^{-(\alpha+1)n}.
\end{multline}

It is easy to check that
\begin{equation}
\frac{\sin\big(\pi[r_2-s]\big)\sin\big(\pi[t-r_1]\big)}{\pi \sin\big(\pi[r_2-r_1]\big)} 
- \frac{\sin\big(\pi[r_1-s]\big)\sin\big(\pi[t-r_2]\big)}{\pi \sin\big(\pi[r_2-r_1]\big)}
=
\frac{\sin\big(\pi[t-s]\big)}{\pi},
\end{equation}
therefore, combining this with  \eqref{eqObsApproxEquality}, \eqref{eqObsApproxEqualityFirstTermBound}, and \eqref{eqObsApproxEqualitySecondTermBound}, we obtain
\begin{equation}
\left|\frac{1}{\obs{\phi}{s}{t}}-\frac{\sin\big(\pi[t-s]\big)}{\pi}\right| 
\leq
e^{10^7 K+60}\, 2^{-(\alpha+1)n} 
\leq e^{10^7 K+60}\, \dist{s}{t}^{\alpha+1}.
\end{equation}
Thus, using $\dist{s}{t}\leq 2^{-R}$, we deduce that
\begin{equation}
\left|\obs{\phi}{s}{t}-\frac{\pi}{\sin\big(\pi[t-s]\big)}\right| 
\leq e^{10^7 K+100} \, \dist{s}{t}^{\alpha-1}.
\end{equation}

\medskip

\textbf{Step 4.} It follows from \eqref{eqLogObsApproxUnifStep2Goal} and continuity of $\obs{\phi}{s}{t}$ in $s$ and $t$ that 
\begin{equation}
\MyHol{\alpha}{2^R}\supset \cap_{n=R}^{\infty} \mathsf{A}_n.
\end{equation}
Combining this with \eqref{eqCapALowerBound}, we finish the proof.
\end{proof}

\section{Proof of Proposition~\ref{prp_action_min_constrained}}
We recall the Proposition~\ref{prp_action_min_constrained}, which we prove in this section.
\begin{customprp}{\ref{prp_action_min_constrained}}
Let $\psi: [t_1, t_2]\to [p_1, p_2]$ be the minimizer of the functional
\begin{equation}
\phi\mapsto \dfrac{1}{2}\displaystyle\int_{t_1}^{t_2} 
\left[\Big(\log \phi'(\tau)\Big)'^{\, 2}-4\pi^2\phi'^{\, 2}(\tau)\right] \d\tau
\end{equation}
on the space $\LogSob(\T)$ under constraints 
\begin{align}
\phi(t_1) &= p_1, \quad \phi(t_2) = p_2,\\
\phi'(t_1) &= q_1, \quad \phi'(t_2) = q_2.
\end{align}
Denote
\begin{equation}
\varkappa := \frac{\pi\sqrt{q_1 q_2}(t_2-t_1)}{\sin\big(\pi(p_2-p_1)\big)}.
\end{equation}
Then for some $ a, b, c, d\in \R$ with $ad-bc= 1$ and any $\tau\in[t_1, t_2]$: 
\begin{enumerate}
\item If $\varkappa>1$,
then
\begin{equation}
\tan\big(\pi\, \psi(\tau)\big) = \frac{a\tan(\lambda\tau) +b}{c \tan(\lambda\tau) +d},
\end{equation}
where $\lambda\in(0, \pi/(t_2 -t_1))$ is such that
\begin{equation}
\frac{\lambda (t_2-t_1)}{\sin\big(\lambda(t_2-t_1)\big)}
= \varkappa.
\end{equation}

\item  If $\varkappa=1$,
then
\begin{equation}
\tan\big(\pi\,\psi(\tau)\big) = \frac{a \,\tau +b}{c \,\tau +d}.
\end{equation}

\item If $\varkappa<1$,
then
\begin{equation}
\tan\big(\pi\,\psi(\tau)\big) = \frac{a\tanh(\lambda\tau) +b}{c \tanh(\lambda\tau) +d},
\end{equation}
where $\lambda>0$ is such that
\begin{equation}
\frac{\lambda (t_2-t_1)}{\sinh\big(\lambda(t_2-t_1)\big)}
=\varkappa.
\end{equation}
\end{enumerate}
\end{customprp}

\begin{proof}

Observe that, for any $\phi\in \DiffAppl \cap \Diff^{\infty}(\T)$,
\begin{equation}
\frac{1}{2}\displaystyle\int_{t_1}^{t_2} 
\left[\Big(\log \phi'(\tau)\Big)'^{\, 2}-4\pi^2\phi'^{\, 2}(\tau)\right] \d\tau 
= 
\frac{\phi''(t_2)}{\phi'(t_2)}-\frac{\phi''(t_1)}{\phi'(t_1)} -\int_{t_1}^{t_2} \Big[\Schw\big(\phi, \tau\big) + 2\pi^2\phi'^{\, 2}(\tau)\Big]\d\tau.
\end{equation}
Fix any $p\in \T\backslash [p_1,  p_2]$, and denote $f = \tan(\pi\,\phi -\pi p -\frac{\pi}{2} )$. Then
\begin{equation}
\int_{t_1}^{t_2} \Big[\Schw\big(\phi, \tau\big) + 2\pi^2\phi'^{\, 2}(\tau)\Big]\d\tau 
= 
\int_{t_1}^{t_2} \Schw\big(f, \tau\big) \d\tau 
=
\frac{f''(t_2)}{f'(t_2)} - \frac{f''(t_1)}{f'(t_1)}
-\frac{1}{2}\int_{a_1}^{a_2} \left(\dfrac{f''(\tau)}{f'(\tau)}\right)^2 \d\tau.
\end{equation}
Since 
\begin{equation}
\frac{f''(\tau)}{f'(\tau)} = \frac{\phi''(\tau)}{\phi'(\tau)} - 2\pi\phi'(\tau)\tan\big(\pi \phi(\tau)-\pi p-\tfrac{\pi}{2}\big),
\end{equation}
we get that
\begin{multline}
\frac{1}{2}\displaystyle\int_{t_1}^{t_2} 
\left[\Big(\log \phi'(\tau)\Big)'^{\, 2}-4\pi^2\phi'^{\, 2}(\tau)\right] \d\tau \\
= 
2\pi\phi'(t_2)\tan\big(\pi \phi(t_2)-\pi p-\tfrac{\pi}{2}\big) - 2\pi\phi'(t_1)\tan\big(\pi \phi(t_1)-\pi p-\tfrac{\pi}{2}\big)+
\frac{1}{2}\int_{t_1}^{t_2} \Big(\log f'(\tau)\Big)'^{\, 2} \d\tau.
\end{multline}
The equality above can be extended to every $\phi\in \DiffAppl\cap \LogSob(\T)$, because $ \DiffAppl \cap \Diff^{\infty}(\T)$ is dense in $\DiffAppl \cap \LogSob$ with respect to $\dist{\cdot}{\cdot}_{\LogSob}$.
Since we keep $\phi(t_1), \phi(t_2), \phi'(t_1), \phi'(t_2)$ fixed, it is sufficient to minimize $\frac{1}{2}\int_{t_1}^{t_2} \big(\log f'(\tau)\big)'^{\, 2} \d\tau$ 
with fixed $f(t_1), f(t_2), f'(t_1), f'(t_2)$.

Denote $g = \log f'$. Notice that
\begin{equation}
\frac{1}{2}\int_{t_1}^{t_2} \left(\frac{f''(\tau)}{f'(\tau)}\right)^2 \d\tau 
= 
\frac{1}{2}\int_{t_1}^{t_2} g'^{\, 2}(\tau) \d\tau.
\end{equation}
Therefore, we want to minimize $\frac{1}{2}\int_{t_1}^{t_2} g'^{\, 2}(\tau) \d\tau$ while keeping $g(t_1), g(t_2), \int_{t_1}^{t_2}e^{g(\tau)}\d\tau$ fixed.

\begin{lmm}\label{lmmApplicationDerivOfLogDiffEq}
There exist $\alpha, \beta, \lambda\in \R$ such that one of the following holds for any $\tau\in [t_1, t_2]$:
\begin{align}
\label{eqlmmApplicationDerivOfLogDiffEq1}
g(\tau) &= \alpha+ \lambda \tau,  
\\
\label{eqlmmApplicationDerivOfLogDiffEq3}
g(\tau) &= \alpha- \log \cosh^2(\beta+\lambda \tau), 
\\
\label{eqlmmApplicationDerivOfLogDiffEq4}
g(\tau) &= \alpha- \log \sinh^2(\beta+\lambda \tau),
\\
\label{eqlmmApplicationDerivOfLogDiffEq5}
g(\tau) &= \alpha- \log \cos^2(\beta+\lambda \tau),
\\
g(\tau) &= - 2\log |\beta+\lambda \tau|.\label{eqlmmApplicationDerivOfLogDiffEq2}
\end{align}
\end{lmm}
\begin{proof}
Using Lagrange multipliers, by varying $\frac{1}{2}\int_{t_1}^{t_2} g'^{\, 2}(\tau) \d\tau$, we get
\begin{equation}\label{eqApplicDiffEqG}
g''(\tau) + \lambda_0 e^{g(\tau)}=0,
\end{equation}
for some $\lambda_0\in \R$.
Below we show that for any given $\lambda_0\in \R, g(t_1)\in \R$, and $g'(t_1)>0$, we can find a solution of the form \eqref{eqlmmApplicationDerivOfLogDiffEq1} -- \eqref{eqlmmApplicationDerivOfLogDiffEq2}, which solves the same Cauchy problem as $g$.
In other words, it also satisfies \eqref{eqApplicDiffEqG} and has the same value and derivative as $g$ at the point $t_1$. 
By uniqueness theorem for differential equations, we will get that $g$ has the form \eqref{eqlmmApplicationDerivOfLogDiffEq1} -- \eqref{eqlmmApplicationDerivOfLogDiffEq2}.

We consider three cases, based on the sign of $\lambda_0$:

\textbf{Case 1.} If $\lambda_0 =0$, then $g$ is a linear function, so for some $\alpha, \lambda\in \R$
\begin{equation}
g(\tau) = \alpha + \lambda \tau.
\end{equation}

\textbf{Case 2.} If $\lambda_0 >0$, then for any $u, v$,
the function
\begin{equation}
F_{u, v}(\tau) = u- \log \cosh^2 \left(v+e^{u/2}\sqrt{\frac{\lambda_0}{2}}\, \tau\right)
\end{equation}
also satisfies the differential equation \eqref{eqApplicDiffEqG}.
Now, we show that we can find such $u, v\in \R$ that $F_{u, v}(t_1) = g(t_1)$ and $F_{u, v}'(t_1) = g'(t_1)$. 
First, to ensure that $F_{u, v}'(t_1) = g'(t_1)$, we take $v=v(u)\in \Compl$ such that 
\begin{equation}
\tanh \left(v+e^{u/2}\sqrt{\frac{\lambda_0}{2}}\, t_1\right)
=
-\frac{1}{2} e^{-u/2}\sqrt{\frac{2}{\lambda_0}} g'(t_1).
\end{equation}
Then, using $\cosh^{-2}(x) = 1- \tanh^2(x)$,
\begin{equation}
-\log {\cosh^2 \left(v+e^{u/2}\sqrt{\frac{\lambda_0}{2}}\, \tau\right)} = \log \left(1-\frac{1}{2\lambda_0}e^{-u}g'^{\, 2}(t_1)\right)
\end{equation}
Now, to ensure that $F_{u, v}(t_1) = g(t_1)$, it remains to choose such $u\in \R$ that
\begin{equation}
\log \left(e^u-\frac{1}{2\lambda_0}g'^{\, 2}(t_1)\right) = g(t_1).
\end{equation}
Notice that we can also choose $v=v(u)$ to be real, since it follows from the equation above that $\left|\frac{1}{2} e^{-u/2}\sqrt{\frac{2}{\lambda_0}} g'(t_1)\right|<1$.
Therefore, by the uniqueness theorem for differential equations we obtain $\tilde{g}_{u, v}(\tau) = g(\tau)$ for all $\tau$.

\textbf{Case 3.} If $\lambda_0 <0$, then for any $u, v$ functions
\begin{align}
F_{u, v}(\tau) &= u-  \log \sinh^2 \left(v+e^{u/2}\sqrt{-\frac{\lambda_0}{2}}\, \tau\right)\\
G_{u, v}(\tau) &= u-  \log \cos^2 \left(v+e^{u/2}\sqrt{-\frac{\lambda_0}{2}}\, \tau\right)\\
H_{v}(\tau) &= - 2 \log \left|v+\sqrt{-\frac{\lambda_0}{2}}\, \tau\right|
\end{align}
also satisfy the differential equation \eqref{eqApplicDiffEqG}.
We consider three subcases:

\textit{Case i.} Suppose that $-\frac{1}{2\lambda_0}g'^{\, 2}(t_1)-e^{g(t_1)}>0$. 
In this case we find $u, v$ such that $F_{u, v}(t_1) = g(t_1)$ and $F_{u, v}'(t_1) = g'(t_1)$. 

First, to ensure that $F_{u, v}'(t_1) = g'(t_1)$, we take $v=v(u)\in \Compl$ such that 
\begin{equation}
\coth \left(v+e^{u/2}\sqrt{-\frac{\lambda_0}{2}}\, t_1\right)
=-\frac{1}{2} e^{-u/2}\sqrt{-\frac{2}{\lambda_0}} g'(t_1).
\end{equation}
Then, using $\sinh^{-2}(x) = \coth^2(x)-1$,
\begin{equation}
-\log {\sinh^2 \left(v+e^{u/2}\sqrt{-\frac{\lambda_0}{2}}\, \tau\right)} = \log \left(-\frac{1}{2\lambda_0}e^{-u}g'^{\, 2}(t_1)-1\right)
\end{equation}
Now, to ensure that $F_{u, v}(t_1) = g(t_1)$, it remains to choose such $u\in \R$ that
\begin{equation}
\log \left(-\frac{1}{2\lambda_0}g'^{\, 2}(t_1)-e^u\right) = g(t_1),
\end{equation}
which can be done because $-\frac{1}{2\lambda_0}g'^{\, 2}(t_1)-e^{g(t_1)}>0$. 
Notice that we can also choose $v=v(u)$ to be real, since it follows from the equation above that $\left|\frac{1}{2} e^{-u/2}\sqrt{-\frac{2}{\lambda_0}} g'(t_1)\right|>1$.
Therefore, by the uniqueness theorem for differential equations, we obtain $F_{u, v}(\tau) = g(\tau)$ for all $\tau$.

\textit{Case ii.} Suppose that $-\frac{1}{2\lambda_0}g'^{\, 2}(t_1)-e^{g(t_1)}<0$.
In this case we find $u, v$ such that $G_{u, v}(t_1) = g(t_1)$ and $G_{u, v}'(t_1) = g'(t_1)$. 

First, to ensure that $G_{u, v}'(t_1) = g'(t_1)$, we take $v=v(u)\in \Compl$ such that 
\begin{equation}
\tan \left(v+e^{u/2}\sqrt{-\frac{\lambda_0}{2}}\, t_1\right)
=\frac{1}{2} e^{-u/2}\sqrt{-\frac{2}{\lambda_0}} g'(t_1).
\end{equation}
Then, using $\cos^{-2}(x) = \tan^2(x)+1$,
\begin{equation}
-\log {\sinh^2 \left(v+e^{u/2}\sqrt{-\frac{\lambda_0}{2}}\, \tau\right)} = \log \left(-\frac{1}{2\lambda_0}e^{-u}g'^{\, 2}(t_1)+1\right)
\end{equation}
Now, to ensure that $G_{u, v}(t_1) = g(t_1)$, it remains to choose $u\in \R$ such that
\begin{equation}
\log \left(-\frac{1}{2\lambda_0}g'^{\, 2}(t_1)+e^u\right) = g(t_1),
\end{equation}
which can be done because $-\frac{1}{2\lambda_0}g'^{\, 2}(t_1)-e^{g(t_1)}<0$. 
Therefore, by the uniqueness theorem for differential equations, we obtain $G_{u, v}(\tau) = g(\tau)$ for all $\tau$.

\textit{Case iii.} Suppose that $-\frac{1}{2\lambda_0}g'^{\, 2}(t_1)-e^{g(t_1)}=0$.
It is easy to check that taking $v=e^{-g(t_1)/2}-\sqrt{-\frac{\lambda_0}{2}} t_1$, we get that $H_{v}(t_1) = g(t_1)$ and $H_{v}'(t_1) = g'(t_1)$. 
Therefore, by the uniqueness theorem for differential equations we obtain $H_{u, v}(\tau) = g(\tau)$ for all $\tau$.
\end{proof}

Therefore, using Lemma \ref{lmmApplicationDerivOfLogDiffEq}, it is easy to show that one of the following holds for some $ a, b, c, d\in \R$ with $ad-bc= \pm 1$, $\lambda>0$ and any $\tau\in[t_1, t_2]$: 
\begin{align}
g(\tau) &= \frac{a\tan(\lambda\tau) +b}{c \tan(\lambda\tau) +d}; \label{eqApplicFFormulaTan}\\
g(\tau) &= \frac{a\tanh(\lambda\tau) +b}{c \tanh(\lambda\tau) +d};\label{eqApplicFFormulaTanh}\\
g(\tau) &= \frac{a \,\tau +b}{c \,\tau +d}.\label{eqApplicFFormulaMob}
\end{align}
Here \eqref{eqApplicFFormulaTan} corresponds to \eqref{eqlmmApplicationDerivOfLogDiffEq3} and \eqref{eqlmmApplicationDerivOfLogDiffEq4}, \eqref{eqApplicFFormulaTanh} corresponds to \eqref{eqlmmApplicationDerivOfLogDiffEq1} and \eqref{eqlmmApplicationDerivOfLogDiffEq2},
and
\eqref{eqApplicFFormulaMob} corresponds to \eqref{eqlmmApplicationDerivOfLogDiffEq5}
Notice that if \eqref{eqApplicFFormulaTan} holds, then $\lambda<\pi/(t_2-t_1)$, since otherwise $g$ is not one to one on $[t_1, t_2]$. 
Moreover, since $g$ is monotonically increasing, we get that $ad-bc = 1$.

Furthermore, notice that 
\begin{equation}
\frac{\pi \sqrt{\psi'(t_1)\psi'(t_2)}(t_2-t_1)}{\sin\Big( \pi\big[\psi(t_2)-\psi(t_1)\big]\Big)}  
= 
\displaystyle\frac{\sqrt{g'(t_1)g'(t_2)}(t_2-t_1)}{g(t_2)-g(t_1)}
\end{equation}
if \eqref{eqApplicFFormulaTan} holds, then
\begin{equation}
\displaystyle\frac{\sqrt{g'(t_1)g'(t_2)}(t_2-t_1)}{g(t_2)-g(t_1)} 
=
\frac{\lambda(t_2-t_1)}{\sin\big(\lambda(t_2-t_1)\big)} \in (1, \infty).
\end{equation}
If \eqref{eqApplicFFormulaTanh} holds, then
\begin{equation}
\displaystyle
\frac{\sqrt{g'(t_1)g'(t_2)}(t_2-t_1)}{g(t_2)-g(t_1)} 
= 
\frac{\lambda(t_2-t_1)}{\sinh\big(\lambda(t_2-t_1)\big)}\in (0, 1).
\end{equation}
If \eqref{eqApplicFFormulaMob} holds, then
\begin{equation}
\displaystyle
\frac{\sqrt{g'(t_1)g'(t_2)}(t_2-t_1)}{g(t_2)-g(t_1)} 
= 1.
\end{equation}
This finishes the proof of the proposition.
\end{proof}

\newpage
\appendix

\section{Appendix}

\begin{stm}\label{stmLogSinTriangle}
For any $x\in (0, \pi/2)$ and $y \in (-x/2, x/2)$, 
\begin{equation}
\Big|\log \sin(x) - \log \sin(x+y) \Big|
\leq \frac{1000\, y}{x}.
\end{equation}
\end{stm}
\begin{proof}
This follows immediately from the fact that
\begin{equation}
\left|\frac{\d}{\d u} \log \sin (u) \right| = \left|\tan(u)\right| \leq \frac{100}{u},
\qquad \text{for } u\in (0, 3\pi/4).
\end{equation}
\end{proof}

\begin{stm} \label{stmArccoshTaylor}
There exists $C>0$ such that for all $x\geq 0$ and $y\in [-\frac{1}{2}, \frac{1}{2}]$, 
\begin{equation}
\arccosh^2\left[\cosh\left(x\right)+y\right]
\geq
x^2 + 2y - C|y|(x^2+|y|),
\end{equation}
and for all $x\geq 10$ and $y\in [-\frac{1}{2}, \frac{1}{2}]$, 
\begin{equation}\label{eq_Arccosh_Taylor_Two}
\left|\arccosh^2\left[\cosh\left(x\right)+y\right]-x^2\right| 
\leq
Cy \, e^{-x/2}.
\end{equation}
\end{stm}

\begin{proof}
Notice that for $u\geq 1$,
\begin{equation}
\frac{\d}{\d u}\arccosh^2\left[u\right]
=
2 \frac{\arccosh\left[u\right]}{\sqrt{u^2-1}},
\end{equation}
where both sides are analytic in $\D_1$.
Inequality \eqref{eq_Arccosh_Taylor_Two} follows since the right-hand side is smaller than $1/\sqrt{u}$ for large $u$.
Moreover, for some $C_1>0$,
\begin{equation}
\left|\frac{\d^2}{\d u^2}\arccosh^2\left[u\right]\right|
\leq  C_1,
\end{equation}
and
\begin{equation}
\frac{\d}{\d u}\arccosh^2\left[u\right]\Big|_{u=1} = 2.
\end{equation}
It is also easy to see that for some $C_2>0$,
\begin{equation}
\left|\frac{\d}{\d u}\arccosh^2\left[u\right] -2\right| 
\leq C_2\, \min\big\{ |u|-1, 1\big\}.
\end{equation}
Taylor expanding $\arccosh^2\left[\cosh\left(x\right)+y\right]$ in $y$ and using the fact that $\min\big\{ \cosh(x)-1, 1\big\}\leq C_3 x^2$ for some $C_3>0$ gives the desired result.
\end{proof}

\begin{stm}\label{stmAppendixIneqSin}
For any $x>0$ and $\tau\in \R$,
\begin{equation}
\left|\frac{\sin\left(xt\right)}{x} - \sin t\right|\leq \frac{1}{6}|x^2-1|t^3.
\end{equation} 
\end{stm}

\begin{proof}
We write
\begin{equation}\label{eqStmAppendixIneqSinIntegral}
\frac{\sin\left(xt\right)}{x} -  \sin t= \int_0^t \big(\cos(x\tau) - \cos(\tau)\big)\d\tau
\end{equation}
and notice that from the inequality $|\sin y|\leq y$, we obtain
\begin{equation}
\big|\cos(x\tau) - \cos(\tau)\big| = 
2\left|\sin\left(\frac{\tau(x+1)}{2}\right)\sin\left(\frac{\tau(x-1)}{2}\right)\right|
\leq \frac{1}{2} |x^2-1| \tau^2,
\end{equation}
which together with \eqref{eqStmAppendixIneqSinIntegral} gives the desired result.
\end{proof}

\begin{stm}\label{stmAppendixSLDerivBounds}
There exists $C>0$ such that for every $\eps \in (0,1/5)$ if M\"{o}bius transformation of the unit circle $\phi \in \SL(2, \R)$ satisfies
\begin{equation}
\left|\phi\left(\tfrac{1}{3}\right) - \phi(0)-\tfrac{1}{3}\right|<\eps, \qquad
 \left|\phi\left(\tfrac{2}{3}\right) - \phi\left(\tfrac{1}{3}\right)-\tfrac{1}{3}\right|<\eps, \qquad 
 \left|\phi\left(1\right) - \phi(\tfrac{2}{3})-\tfrac{1}{3}\right|<\eps, 
\end{equation}
then for any $t\in \T$,
\begin{align}
\big|\phi'(t)-1\big|<C\eps,\label{stmAppendixSLDerivBoundsOne}\\
\big|\phi''(t)\big|<C\eps\label{stmAppendixSLDerivBoundsTwo}.
\end{align}
\end{stm}

\begin{proof}
By rotational invariance of the properties in question, it is sufficient to prove the statement only for $t\in[1/3, 2/3]$ and we can assume that $\phi(0) = 0$.
Therefore,
\begin{equation}
\left|\phi\left(\tfrac{1}{3}\right) - \tfrac{1}{3}\right| \leq\eps, \qquad 
\left|\phi\left(\tfrac{2}{3}\right) - \tfrac{2}{3}\right| \leq\eps.
\end{equation}

Denote
\begin{equation}
F(\tau) = -\cot\left(\pi\, \phi\left(\frac{\arccot(-\tau)}{\pi}\right)\right).
\end{equation}
Notice that $F(\tau)$ is a M\"{o}bius transformation of the upper half-plane, that is
\begin{equation}
F(\tau) = \frac{a\tau+b}{c\tau+d}, \qquad \text{for some } a, b, c, d\in\R, \text{with } ad-bc = 1.
\end{equation}

From the fact that $\phi(0) = 0$, it follows that $F(\infty) = \infty$. Thus, $F$ is linear
\begin{equation}
F(\tau) = a\tau+b, \qquad \text{for some } a>0, \text{ and } b\in\R.
\end{equation}
Moreover, for some $C_1>0$, we have
\begin{equation}
\left|F\left(\frac{1}{\sqrt{3}}\right)-\frac{1}{\sqrt{3}}\right| < C_1\eps,\qquad
\left|F\left(-\frac{1}{\sqrt{3}}\right)+\frac{1}{\sqrt{3}}\right|<C_1 \eps.
\end{equation}
Therefore,
\begin{equation}
\left|\frac{a-1}{\sqrt{3}}+b\right|<C_1\eps, \qquad \left|-\frac{a-1}{\sqrt{3}}+b\right|<C_1\eps.
\end{equation}
Thus, for some $C_2>0$, we have 
\begin{equation}
|a-1|<C_2\eps, \qquad |b|<C_2\eps.
\end{equation}
Therefore, for some $C_3>0$ and each $\tau\in [-1/\sqrt{3}, 1/\sqrt{3}]$, 
\begin{equation}\label{eqStmAppendixSLDerivBoundsChangeOfVariableBound}
|F(\tau)-\tau|<C_3\eps, \qquad |F'(\tau)-1|<C_3\eps.
\end{equation}

From the equality
\begin{equation}
\phi(t) = \frac{1}{\pi}\arccot \Big(- F\big[ -\cot(\pi t) \big] \Big),
\end{equation}
we obtain
\begin{equation}\label{eqStmAppendixSLPhiFormula}
\phi'(t) = \arccot'\Big(-F\big[-\cot(\pi t)\big]\Big)F'\big[-\cot(\pi t)\big]\cot'(\pi t).
\end{equation}
Using \eqref{eqStmAppendixSLDerivBoundsChangeOfVariableBound}, 
the fact that
\begin{equation}\label{eqStmAppendixSLArccotCotProd}
\arccot'\big[\cot(\pi t)\big]\cot'(\pi t) = 1
\end{equation}
and the fact that
\begin{equation}\label{eqStmAppendixSLCotImage}
\forall t\in[1/3, 2/3]:\qquad -\cot(\pi t) \in [-1/\sqrt{3}, 1/\sqrt{3}]
\end{equation}
we finish the proof of \eqref{stmAppendixSLDerivBoundsOne}.

In order to prove \eqref{stmAppendixSLDerivBoundsTwo}, we notice that
\begin{multline}
\phi''(t) = \pi\arccot''\Big(-F\big[-\cot(\pi t)\big]\Big)\Big(F'\big[-\cot(\pi t)\big]\cot'(\pi t)\Big)^2
\\
+\pi\arccot'\Big(-F\big[-\cot(\pi t)\big]\Big)F'\big[-\cot(\pi t)\big]\cot''(\pi t)
.
\end{multline}
Combining this with \eqref{eqStmAppendixSLDerivBoundsChangeOfVariableBound}, \eqref{eqStmAppendixSLCotImage}, and the observation that
\begin{equation}
0 = \Big(\arccot'\big[\cot(\pi t)\big]\cdot \cot'(\pi t)\Big)' 
= \pi\arccot''\big[\cot(\pi t)\big]\cdot \big(\cot'(\pi t)\big)^2 + \pi\arccot'\big[\cot(\pi t)\big]\cdot \cot''(\pi t)
\end{equation}
we finish the proof of \eqref{stmAppendixSLDerivBoundsTwo}.
\end{proof}

\section*{Acknowledgements}
The author thanks Roland Bauerschmidt, James Norris, and Peter Wildemann for many helpful discussions and careful proofreading.

\bibliography{Schwarzian}

\end{document}